%% file: string-top-based-loop-space.tex
\documentclass[12pt,letterpaper]{amsart}

\usepackage[margin=1in]{geometry}
\usepackage{hyperref}
\usepackage[arrow,matrix,tips,curve]{xy} 
\usepackage{paralist} 
\usepackage{textcomp} 
\usepackage{url}
\usepackage{macros}

\reversemarginpar

\usepackage{mathpazo}
\begin{document}

\title{String Topology and the Based Loop Space}

\author{Eric J. Malm}
\address{Eric J. Malm, Simons Center for Geometry and Physics, Stony Brook University, Stony Brook, NY 11794-3636, USA}
\email{emalm@scgp.stonybrook.edu}
\thanks{Partly supported by a National Defense Science and Engineering Graduate Fellowship.}

\date{\today}

\input{abstract.tex}

\maketitle

\input{ch-intro.tex}

\input{ch-bg-revised.tex}

\input{ch-hh.tex}

\input{ch-bv.tex}

\appendix

\input{app-alg.tex}

\bibliographystyle{plain}
\bibliography{math-master}

\end{document}

%% file: abstract.tex
\begin{abstract}
For $M$ a closed, connected, oriented manifold, we obtain the Batalin-Vilkovisky (BV) algebra of its string topology through homotopy-theoretic constructions on its based loop space. In particular, we show that the Hochschild cohomology of the chain algebra $C_* \Omega M$ carries a BV algebra structure isomorphic to that of the loop homology $\mathbb{H}_*(LM)$. Furthermore, this BV algebra structure is compatible with the usual cup product and Gerstenhaber bracket on Hochschild cohomology. To produce this isomorphism, we use a derived form of Poincar\'e duality with $C_*\Omega M$-modules as local coefficient systems, and a related version of Atiyah duality for parametrized spectra connects the algebraic constructions to the Chas-Sullivan loop product.
\end{abstract}

%% file: ch-intro.tex
\section{Introduction}\label{ch:intro}

\input{sec-intro-statement.tex}

\input{sec-intro-outline.tex}

%% file: sec-intro-statement.tex
\stepcounter{subsection}

String topology, as initiated by Chas and Sullivan in their 1999 paper~\cite{chas-sullivan:1999}, is the study of algebraic operations on $H_*(LM)$, where $M$ is a closed, smooth, oriented $d$-manifold and $LM = \Map(S^1, M)$ is its space of free loops. They show that, because $LM$ fibers over $M$ with fiber the based loop space $\Omega M$ of $M$, $H_*(LM)$ admits a graded-commutative \term{loop product}
\[
  \circ: H_p(LM) \otimes H_q(LM) \to H_{p + q - d}(LM)
\]
of degree $-d$. Geometrically, this loop product arises from combining the intersection product on $H_*(M)$ and the Pontryagin or concatenation product on $H_*(\Omega M)$. Writing $\mathbb{H}_*(LM) = H_{* + d}(LM)$ to regrade $H_*(LM)$, the loop product makes $\mathbb{H}_*(LM)$ a graded-commutative algebra. Chas and Sullivan describe this loop product on chains in $LM$, but because of transversality issues they are not able to construct the loop product on all of $C_*LM$ this way. Cohen and Jones instead give a homotopy-theoretic description of the loop product in terms of a ring spectrum structure on a generalized Thom spectrum $LM^{-TM}$. 

$H_*(LM)$ also admits a degree-$1$ operator $\Delta$ with $\Delta^2 = 0$, coming from the $S^1$-action on $LM$ rotating the free loop parameterization. Furthermore, the interaction between $\Delta$ and $\circ$ makes $\mathbb{H}_*(LM)$ a Batalin-Vilkovisky (BV) algebra, or, equivalently, an algebra over the homology of the framed little discs operad. Consequently, it is also a Gerstenhaber algebra, an algebra over the homology of the (unframed) little discs operad, via the loop product $\circ$ and the \term{loop bracket} $\{-, -\}$, a degree-$1$ Lie bracket defined in terms of $\circ$ and $\Delta$.

Such algebraic structures arise in other mathematical contexts. For example, if $A$ is a differential graded algebra, its Hochschild homology $HH_*(A)$ has a degree-$1$ Connes operator $B$ with $B^2 = 0$, and its Hochschild cohomology $HH^*(A)$ is a Gerstenhaber algebra under the Hochschild cup product $\cup$ and the Gerstenhaber Lie bracket $[-, -]$. Consequently, it is natural to ask whether these constructions recover some of the structure of string topology for a choice of algebra $A$ related to $M$. Two algebras that arise immediately as candidates are $C^*M$, the differential graded algebra of cochains of $M$ under cup product, and $C_*\Omega M$, the algebra of chains on the based loop space $\Omega M$ of $M$, with product induced by the concatenation of based loops.

In the mid-1980s, Goodwillie and Burghelea and Fiedorowicz independently developed the first result of this form~\cite{burghelea-fiedorowicz:1986, goodwillie:1985}, showing an isomorphism between $H_*(LX)$ and $HH_*(C_*\Omega X)$ for a connected space $X$ that takes $\Delta$ to the $B$ operator. Shortly after this result, Jones used a cosimplicial model for $LM$ to show an isomorphism between $H^*(LX)$ and $HH_*(C^*X)$ when $X$ is simply connected, taking a cohomological version of the $\Delta$ operator to $B$~\cite{jones:1987}. 

With the introduction of string topology, similar isomorphisms relating the loop homology $\mathbb{H}_*(LM)$ of $M$ to the Hochschild cohomologies $HH^*(C^*M)$ and $HH^*(C_*\Omega M)$ were developed. One such family of isomorphisms arises from variations on the Jones isomorphism, and so also requires $M$ to be simply connected. More closely reflecting the Burghelea-Fiedorowicz--Goodwillie perspective, Abbaspour, Cohen, and Gruher~\cite{abbaspour-cohen-gruher:2008} instead show that, if $M$ is a $K(G, 1)$ manifold for $G$ a discrete group, then there is an isomorphism of graded algebras between $\mathbb{H}_*(LM)$ and $H^*(G, kG^c)$, the group cohomology of $G$ with coefficients in the group ring $kG$ with the conjugation action. Vaintrob~\cite{vaintrob:2007} notes that this is also isomorphic to $HH^*(kG)$ and shows that, when $k$ is a field of characteristic $0$, $HH^*(kG)$ admits a BV structure isomorphic to that of string topology. 

Our main result is a generalization of this family of results, replacing the group ring $kG$ with the chain algebra $C_*\Omega M$. When $M = K(G, 1)$, $\Omega M \simeq G$, so $C_*\Omega M$ and $kG = C_*G$ are equivalent algebras.

\begin{thm}\label{thm:main-st-hh}
Let $k$ be a commutative ring, and let $X$ be a $k$-oriented, connected Poincar\'e duality space of dimension $d$. Poincar\'e duality, extended to allow $C_*\Omega X$-modules as local coefficients, gives a sequence of weak equivalences inducing an isomorphism of graded $k$-modules
\[
  D: HH^*(C_*\Omega X) \to HH_{* + d}(C_*\Omega X).
\]
Pulling back $-B$ along $D$ gives a degree-$1$ operator $-D^{-1}BD$ on $HH^*(C_*\Omega X)$. This operator interacts with the Hochschild cup product to make $HH^*(C_*\Omega X)$ a BV algebra, where the induced bracket coincides with the usual bracket on Hochschild cohomology.

When $X$ is a manifold as above, the composite of $D$ with the Goodwillie isomorphism between $HH_*(C_*\Omega X)$ and $H_*(LX)$ gives an isomorphism $HH^*(C_*\Omega X) \isom \mathbb{H}_*(LX)$ taking this BV algebra structure to that of string topology.
\end{thm}

We produce the $D$ isomorphism in Theorem~\ref{thm:hh-additive-isom}, and we establish the BV algebra structure on $HH^*(C_*\Omega M)$ and its relation to the string topology BV algebra in Theorems~\ref{thm:bv-struct-hh}, and~\ref{thm:bv-isom-hh-string-top}. Since the $D$ isomorphism ultimately comes from Poincar\'e duality with local coefficients, this result also allows us to see more directly that the Chas-Sullivan loop product comes from the intersection product on the homology of $M$ with coefficients taken in $C_*\Omega M$ with the loop-conjugation action.

%% file: sec-intro-outline.tex
We summarize the structure of the rest of this document. In Section~\ref{ch:bg}, we provide background and preliminary material for our comparison of string topology and Hochschild homology. We state the basic properties of the singular chains $C_*X$ of a space $X$, including the algebra structure when $X$ is a topological monoid. We also develop the notions of $\Ext$, $\Tor$, and Hochschild homology and cohomology over a differential graded algebra $A$ in terms of a model category structure on the category of $A$-modules, and we use two-sided bar constructions as models for this homological algebra. Via Rothenberg-Steenrod constructions, we relate these algebraic constructions to the topological setting. Additionally, we state the key properties of the loop product $\circ$ and BV operator $\Delta$ in string topology, and we survey previous connections between the homology of loop spaces and Hochschild homology and cohomology. Finally,  we review the extended or ``derived'' Poincar\'e duality we use above, which originates in work of Klein~\cite{klein:1999b} and of Dwyer, Greenlees, and Iyengar~\cite{dwyer-greenlees-iyengar:2006}. In this setting, we broaden the notion of local coefficient module for $M$ to include modules over the DGA $C_*\Omega M$, instead of simply modules over $\pi_1 M$, and we show that Poincar\'e duality for $\pi_1 M$-modules implies Poincar\'e duality for this wider class of coefficients.  

In Section~\ref{ch:hh}, we relate the Hochschild homology and cohomology of $C_*\Omega M$ to this extended notion of homology and cohomology with local coefficients, where the coefficient module is $C_*\Omega M$ itself with an action coming from loop conjugation. In fact, there are several different models of this adjoint action that are convenient to use in different contexts, and in order to switch between them we must employ some technical machinery involving morphisms of $A_\infty$-modules between modules over an ordinary DGA. In any case, this result combines with Poincar\'e duality to establish the isomorphism $D$ above, coming from a sequence of weak equivalences on the level of chain complexes.

Section~\ref{ch:bv} relates the BV structure of the string topology of $M$ to the algebraic structures present on the Hochschild homology and cohomology of $C_*\Omega M$. In order to do so, we must engage with a spectrum-level, homotopy-theoretic description of the Chas-Sullivan loop product. We show that the Thom spectrum $LM^{-TM}$ and the topological Hochschild cohomology of $S[\Omega M]$, the suspension spectrum of $\Omega M$, are equivalent as ring spectra, using techniques in fiberwise spectra from Cohen and Klein~\cite{cohen-klein:2009}. We recover the chain-level equivalences established earlier by smashing with the Eilenberg-Mac Lane spectrum $Hk$ and passing back to the equivalent derived category of chain complexes over $k$. 

We then show that the pullback $-D^{-1}BD$ of the $B$ operator to $HH^*(C_* \Omega M)$ forms a BV algebra structure on $HH^*(C_* \Omega M)$, and that this structure coincides with that of string topology. We do this by establishing that $D$ is in fact given by a Hochschild cap product against a fundamental class $z \in HH_d(C_* \Omega M)$, for which $B(z) = 0$. These two conditions allow us to apply an algebraic argument of Ginzburg~\cite{ginzburg:2006}, with some sign corrections by Menichi~\cite{menichi:2009a}, to establish this BV algebra structure.  

Appendix~\ref{app:algebra} contains some conventions regarding chain complexes, coalgebras, and adjoint actions for modules over Hopf algebras. It also contains  our working definitions of $A_\infty$ algebras and $A_\infty$ modules and how they relate to two-sided bar constructions and $\Ext$ and $\Tor$. 

This document contains results from the author's Stanford University Ph.D.\ thesis, and we wish to thank our advisor, Ralph Cohen, for his insight, guidance, and patience. We also wish to thank Gunnar Carlsson, John Klein, and Dennis Sullivan for many useful discussions.

%% file: ch-bg-revised.tex
\section{Background and Preliminaries}\label{ch:bg}

\input{sec-bg-rev-chain-cxs.tex}

\input{sec-bg-rev-hom-alg.tex}

\input{sec-bg-rev-bar-rs.tex}

\input{sec-bg-rev-loop-spaces.tex}

\input{sec-pd-rev.tex}

%% file: sec-bg-rev-chain-cxs.tex
\subsection{Singular Chain Complex Conventions}

We briefly state our conventions regarding simplicial objects and singular chain complexes. Let $\Delta$ denote the \term{simplicial category}, with objects ordered sequences $[n] = \{ 0 < 1 < \dotsb < n \}$ and morphisms order-preserving maps. Let $\Delta^{\bullet}: \Delta \to \Top$ be the canonical cosimplicial space of geometric simplices, with coface maps $d^i$ and codegenerancy maps $s^i$.

Throughout, $k$ will denote a fixed commutative ring, and $\Ch(k)$ the category of unbounded chain complexes of $k$-modules. For a given topological space $X$, let $S_*X$ denote the singular complex of $X$, with face maps $d_i$ and degeneracy maps $s_i$. Applying the Dold-Kan correspondence yields $C_*(X; k)$, the \term{normalized singular chain complex of $X$}, with total differential $d = \sum_{i = 0}^n (-1)^i d_i$ the signed sum of the face maps $d_i$. When $k$ is understood, we also write it as $C_*X$. Note that $C_*(\text{pt}) \isom k$.

The classical Eilenberg-Zilber equivalences relate $C_*(X \times Y)$ and $C_*(X) \otimes C_*(Y)$:

\begin{defn}
Define the \term{Alexander-Whitney map} $AW: C_*(X \times Y) \to C_*X \otimes C_*Y$ by
\[
  AW(\phi) = \sum_{i = 0}^n d_{n-i+1} \dotsm d_{n-1} d_n \pi_X\phi \otimes d_0^{n-i} \pi_Y\phi.
\]
Similarly, define the \term{Eilenberg-Zilber map} $EZ: C_*X \otimes C_*Y \to C_*(X \times Y)$ on simplices $\rho: \Delta^n \to X$ and $\sigma: \Delta^m \to Y$ by
\[
  EZ(\rho \otimes \sigma) = \sum_{(\mu, \nu) \in S_{n, m}} (-1)^{\epsilon(\mu, \nu)} (s_{\nu(n)} \dotsm s_{\nu(1)} \rho, s_{\mu(m)} \dotsm s_{\mu(1)} \sigma),
\]
where $S_{n,m}$ is the set of $(n, m)$-shuffles in $S_{n + m}$, and $\epsilon: S_{n + m} \to \{ \pm 1 \}$ is the sign homomorphism. Note that the $\mu(i)$ and $\nu(j)$ values together range from $0$ through $n + m - 1$.
\end{defn}

See \cite[p.~55]{felix-halperin-thomas:2001} for standard facts about $EZ$ and $AW$, including naturality, associativity, and compatibility with the symmetry maps $t$ in $\Top$ and $\tau$ in $\Ch(k)$. The other key property of these maps is that they are homotopy equivalences: in fact, $AW \circ EZ = \id$, and $EZ \circ AW$ is homotopic to $\id$ by a natural chain homotopy $H$.

It is then standard 
that for any space $X$, $C_*X$ is a counital differential graded coalgebra (DGC) via $\Delta = AW \circ C_*\delta$, and if $(X, m)$ is a topological monoid, then $\mu = C_*m \circ EZ$ makes $C_*X$ a differential graded Hopf algebra.

\subsection{Topological Groups and Singular Complexes}

We focus extensively on constructions involving $\Omega X$, the based loop space of $X$, and consequently want to have models for it that are compatible with the chain-level and spectrum-level constructions we will employ. $\Omega X$ is an $A_\infty$-monoid with a homotopy inverse, and, following Klein~\cite{klein:2001}, we will use a homotopy equivalent topological group model for $\Omega X$. Burghelea and Fiedorowicz \cite[p.~311]{burghelea-fiedorowicz:1986} have one such construction utilizing the Moore loops $MX$ of $X$ and May's categorical two-sided bar construction~\cite{may:1972}, and the Kan loop group $\tilde{G}_\bullet(K)$ of a simplicial set $K$~\cite{goerss-jardine:1999,kan:1958b} is another related construction.

Suppose $G$ is a general topological group, with inverse map $i: G \to G$. From above, $C_*G$ is a DG Hopf algebra, and we discuss the role of $C_*i$ as an antipode map.

\begin{prop} \label{prop:inverse-almost-antipode}
Let $S = C_*i$. Then $S$ is an algebra anti-automorphism of $C_*G$, $S^2 = \id$, and $\Delta (\id \otimes S) \mu \simeq \eta \epsilon$ (similarly for $S \otimes \id$).
\begin{proof}
The anti-automorphism identity, $S \circ \mu = \mu \circ \tau \circ (S \otimes S)$, follows from $i \circ m = m \circ t \circ (i \times i)$ and $C_*t \circ EZ = EZ \circ \tau$. Since $i^2 = \id$, $(C_*i)^2 = C_* i^2 = \id$, so $C_*i$ is an involution of $C_*G$.

Finally, the antipode diagrams for $\id \otimes S$ (shown below) and $S \otimes \id$ commute up to chain homotopy, using the chain homotopy $H$ from $EZ \circ AW$ to $\id$:
\[
	\xymatrix{
  	& & C_*G^{\otimes 2} \ar[rr]^{\id \otimes S} & & C_*G^{\otimes 2} \ar[dr]^{EZ} \\
	& C_*(G \times G) \ar[ur]^{AW} \ar[rr]^{C_*(\id \times i)} & & C_*(G \times G) \ar[ur]^{AW} \ar[rr]^{\id} \ar[dr]^{C_*m} & & C_*(G \times G) \ar[dl]_{C_*m} \\
	C_*G \ar[ur]^{C_*\Delta} \ar[rr]^{\epsilon} & & k \ar[rr]^{\eta} & & C_*G \\
	}
\]
\end{proof}
\end{prop}

We will use these structures in Section~\ref{ch:hh} to express the Hochschild homology and cohomology of $C_*\Omega X$ in terms of $\Ext$ and $\Tor$ over $C_*\Omega X$ itself.

%% file: sec-bg-rev-hom-alg.tex
\subsection{Model Categories and Homological Algebra}\label{sec:dgha}

\subsubsection{Model Categories of $A$-modules}

As stated above, we seek to express various features of string topology in terms of the homological algebra of $A$-modules for various choices of DGAs $A$. Consequently, we would like to have a suitable model structure on the category $\modulecategory{A}$ of $A$-modules. In fact, when $A = C_*G$ for $G$ a topological monoid, $A$ is cofibrant in the projective model structure on chain complexes of $k$-modules, and so by results of Schwede and Shipley~\cite[Lemma~2.3]{schwede-shipley:2000} its category of modules has a cofibrantly generated model category structure, which we cover quickly below.

Recall~\cite[\S2.2]{hovey:1999} that $\Ch(k)$, the category of unbounded chain complexes over $k$, has a cofibrantly generated model structure with the generating cofibrations the inclusions $i_n: S^{n - 1} \to D^n$ and the generating trivial cofibrations the $j_n: 0 \to D^n$. Here, $S^n$ is $k$ in degree $n$ and $D^n$ is $\id: k \to k$ in degrees $n$ and $n - 1$. Let $I_A$ and $J_A$ be the images of $I$ and $J$ under the endofunctor $F_A = A \otimes -$. Using the basic properties of monoidal model categories to check the hypotheses of the lemma yields the following result.

\begin{prop}
Suppose that $A$ is a cofibrant object in $\Ch(k)$. Then $\modulecategory{A}$ has a cofibrantly generated model category structure with $I_A$ as the set of generating cofibrations and $J_A$ as the set of generating trivial cofibrations. A morphism in $\modulecategory{A}$ is a weak equivalence or a fibration if the underlying morphism of chain complexes is one.
\end{prop}

This model category structure on $\modulecategory{A}$ is also related to the notion of semifree extensions and $A$-semifree modules, as presented by F\'elix, Halperin, and Thomas~\cite[\S 2]{felix-halperin-thomas:1995}. In particular, the semifree extensions coincide with the $I_A$-cellular maps, and cofibrations coincide with retracts of these semifree extensions. In particular, many  results about semifree extensions also apply to cofibrations, and specialize as follows:

\begin{prop}\label{prop:cof-pres-wk-equivs}
Suppose $P$ and $Q$ are cofibrant (left or right, as is appropriate) $A$-modules.
\begin{compactenum}[(\itshape a\/\upshape)]
	\item $P \otimes_A -$ and $\Hom_A(P, -)$ are exact functors from $\modulecategory{A}$ to $\Ch(k)$ and preserve all weak equivalences in $\modulecategory{A}$.
	
	\item If $f: P \to Q$ is a weak equivalence in $\modulecategory{A}$, then $f \otimes_A M$ and $\Hom_A(f, M)$ are also weak equivalences for all $A$-modules $M$.
\end{compactenum}
\end{prop}

With these foundations in place, $\Ext_A^*$ and $\Tor_A^*$ arise naturally as the homology groups of the derived versions of $\otimes_A$ and $\Hom_A$. To be specific, let $Q$ and $R$ denote cofibrant- and fibrant-replacement functors for $\modulecategory{A}$. The the left derived tensor product $\otimes_A^L$ and right derived $\Hom$ complex $R\Hom_A$ are given by $M \otimes_A^L N = QM \otimes_A QN$ and $R\Hom_A(M, P) = \Hom_A(QM, RP)$. By the convenient properties of cofibrant $A$-modules, it suffices to replace only one of $M$ and $N$ when forming $M \otimes_A^L N$.

Since fibrations of $A$-modules are precisely the levelwise surjections, every $A$-module is fibrant (so $\id$ may be taken as $R$). Since weak equivalences between fibrant-cofibrant objects are homotopy equivalences, and since $Q$ preserves weak equivalences, we have the following:

\begin{prop}\label{prop:cof-wk-equiv-htpy-equiv}
If $M$, $N$ are cofibrant $A$-modules connected by a zigzag of weak equivalences, there is a homotopy equivalence $h: M \to N$ homotopic to this zigzag in $\Ho \modulecategory{A}$.
\begin{proof}
$Q$ lifts this zigzag to one of homotopy equivalences, and the maps $QM \to M$ and $QN \to N$ are also homotopy equivalences.
\end{proof}
\end{prop}

%% file: sec-bg-rev-bar-rs.tex
\subsubsection{Bar constructions, \texorpdfstring{$\Ext$}{Ext}, and \texorpdfstring{$\Tor$}{Tor}}

Bar constructions often provide a convenient construction of cofibrant replacements, and hence of complexes representing $\Ext$ and $\Tor$, and are therefore used extensively in the literature~\cite{burghelea-fiedorowicz:1986,felix-menichi-thomas:2005,goodwillie:1985,jones:1987,loday:1992}. We state our conventions regarding them, which are general enough to perform in a general monoidal category.

\begin{defn}
Let $(\category{C}, \otimes, I)$ be a monoidal category and let $(A,\mu, \eta)$ be a monoid in $\category{C}$. Given a right $A$-module $M$ and a left $A$-module $N$, define the \term{(two-sided) bar construction} to be the simplicial $\category{C}$-object $B_\bullet(M, A, N)$, with $B_n(M, A, N) = M \otimes A^{\otimes n} \otimes N$, $k \geq 0$. The face and degeneracy maps $d_i$ and $s_i$ are given by
\begin{align*}
  s_i &= \id^{\otimes i + 1} \otimes \eta \otimes \id^{\otimes n + 1 - i}, &
  d_i &= \begin{cases}
          a_M \otimes \id^n, & i = 0, \\
		  \id^{\otimes i} \otimes \mu \otimes \id^{\otimes n - i}, & 1 \leq i \leq n - 1, \\
		  \id^{\otimes n} \otimes a_N, & i = n.
         \end{cases} 
\end{align*}
\end{defn}

When $\category{C}$ is $\Ch(k)$ and $A$ is a DGA, then $m[a_1 \mid \dotsb \mid a_n]n$ is often written for a simple element $m \otimes a_1 \otimes  \dotsb \otimes a_n \otimes n \in B_n(M, A, N)$. Each $B_n(M, A, N)$ is itself a chain complex, with the differential given by the Leibniz rule (with appropriate Koszul signs introduced).

We pass from this simplicial object to a single chain complex via an algebraic version of geometric realization, following Shulman~\cite[\S 12]{shulman:2006}.

\begin{defn}\label{defn:geom-real-chain-cx}
Let $X_\bullet$ be a simplicial chain complex, and let $\Delta^\bullet$ denote the standard cosimplicial simplicial set of simplices, with $\Delta^n_m = \Delta([m], [n])$. Then $N(k\Delta^\bullet)$ is the cosimplicial chain complex obtained by applying the Dold-Kan normalization functor $N$ levelwise. Define the \term{geometric realization} of $X_\bullet$ to be the coend $|X_\bullet| = N(k\Delta^\bullet) \otimes_{\Delta^{\op}} X_\bullet$, and the \term{thick realization} of $X_\bullet$ to be $\| X_\bullet \| = k\Delta^\bullet \otimes_{\Delta^{\op}} X_\bullet$.

Define $B(M, A, N) = \| B_\bullet(M, A, N) \|$ and $\bar{B}(M, A, N) = | B_\bullet(M, A, N) |$.
\end{defn}

In fact, treating $B_\bullet(M, A, N)$ as a chain complex of chain complexes $B_*(M, A, N)$ via the Moore complex differential, $d_s = \sum_{i = 0}^k (-1)^i d_i$, $B(M, A, N)$ is isomorphic to the total complex $\Tot(B_*(M, A, N))$.

Furthermore, since $A \otimes -$ commutes with this coend, if $M$ is also a left $A$-module, then so is $B(M, A, N)$. In particular, $B(A, A, N)$ is a left $A$-module, with a natural trivial fibration $q_N$ mapping to $N$. If $N$ is also a cofibrant chain complex, $B(A, A, N)$ is a cofibrant $A$-module, and so these bar constructions give combinatorial models for $\Ext$ and $\Tor$:
\[
  \Ext_A^*(M, N) \isom H_*(\Hom_A(B(A, A, M), N))
  \quad \text{and} \quad
  \Tor^A_*(P, M) \isom H_*(B(P, A, M))
\]

\subsubsection{Hochschild Homology and Cohomology}

Hochschild homology and cohomology provide an algebraic homology and cohomology theory for bimodules over a DGA $A$. We formulate them in terms of $\Ext$ and $\Tor$ over $A^e = A \otimes A^{\op}$, the enveloping algebra of $A$. $A$-$A$-bimodules canonically have both left and right $A^e$-module structures.

\begin{defn}
Given a DGA $A$ and an $A$-$A$-bimodule $M$, define the \term{Hochschild homology} of $A$ with coefficients in $M$ to be $HH_*(A, M) = \Tor^{A^e}_*(M, A)$, treating $M$ canonically as a right $A^e$-module and $A$ as a left $A^e$-module. Similarly, define the \term{Hochschild cohomology} of $A$ with coefficients in $M$ to be $HH^*(A, M) = \Ext_{A^e}^*(A, M)$, where $M$ is now canonically a left $A^e$-module.

When $M = A$, we write $HH_*(A)$ for $HH_*(A, A)$ and $HH^*(A)$ for $HH^*(A, A)$.
\end{defn}

From above, $B(A, A, A)$ is an $A$-$A$-bimodule, and hence canonically a left $A^e$-module. Since $A$ is assumed to be cofibrant in $\Ch(k)$, $B(A, A, A)$ is a cofibrant $A^e$-module, weakly equivalent to $A$.

\begin{defn}
Define the \term{Hochschild (co)chains} of $A$ with coefficients in $M$ to be
\[
  CH_*(A, M) = M \otimes_{A^e} B(A, A, A)
  \quad \text{and} \quad
  CH^*(A, M) = \Hom_{A^e}(B(A, A, A), M).
\]
\end{defn}

In the literature, $CH_n(A, M)$ is often expressed as $M \otimes A^{\otimes n}$, which is isomorphic to $M \otimes_{A^e} B_n(A, A, A)$. In this form, when $M = A$, the cyclic permutation of the $n + 1$ factors of $A$ gives rise to a degree-$1$ operation $B$ on $HH_*(A)$, originally due to Connes~\cite{loday:1992}.

The Hochschild cohomology $HH^*(A)$ also admits natural operations, first due to Gerstenhaber~\cite{gerstenhaber:1963}. The homogeneous complex $\Hom_{A^e}(B_n(A, A, A), A)$ is naturally isomorphic to $\Hom_k(A^{\otimes n}, A)$. Given a $p$-cochain $f: A^{\otimes p} \to A$ and a $q$-cochain $g: A^{\otimes q} \to A$, their cup product $f \cup g$ is $\mu \circ (f \otimes g): A^{\otimes (p + q)} \to A$, and descends to a graded-commutative cup product on $HH^*(A)$. This product is equivalent to the Yoneda product on $R\Hom_{A^e}(A, A)$. Additionally, a commutator-like construction gives a degree-$1$ Lie bracket $[f,g]: A^{\otimes (p + q - 1)} \to A$, and this bracket interacts with the cup product to make $HH^*(A)$ a Gerstenhaber algebra.

Finally, $HH_*(A, M)$ is a right module for the algebra $HH^*(A)$, with the action coming from the action of $R\Hom_{A^e}(A, A)$ on $M \otimes_{A^e}^L A$. This action can also be formulated combinatorially on the cochain level, as with the $B$ operator, the cup product, and the bracket. When $M = A$, this cap product is part of a calculus structure on $(HH^*(A), HH_*(A))$ that formalizes the interaction of differential forms and (poly)vector fields on a manifold~\cite{tamarkin-tsygan:2005}. 

\subsubsection{Rothenberg-Steenrod constructions}

In the case where $A = C_*G$ for $G$ a topological monoid, F\'elix et al.~\cite{felix-halperin-thomas:1995} determine several results which generalize the Rothenberg-Steenrod spectral sequence~\cite[\S 7.4]{mccleary:2001}\cite{rothenberg-steenrod:1965} to equivalences of chain complexes and of differential graded coalgebras, phrased in terms of bar constructions. The main results of consequence to us are as follows:

\begin{thm}
For any path connected space $X$, $C_*X$ and the bar construction $B(k, C_*\Omega X, k)$ are weakly equivalent as differential graded coalgebras.
\end{thm}

\begin{prop} \label{prop:chains-borel-construction}
Let $G$ be a topological group and let $F$ be a right $G$-space. Then $C_*(F \times_G EG)$ and $B(C_*F, C_*G, k)$ are weakly equivalent as differential graded coalgebras. 
\end{prop}

%% file: sec-bg-rev-loop-spaces.tex
\subsection{String Topology and Hochschild Constructions}

\subsubsection{String Topology Operations}

We describe some of the conventions and fundamental operations in string topology. Let $M$ be a closed, smooth, $k$-oriented manifold of dimension $d$, and let $LM = \Map(S^1, M)$ be the space of free loops in $M$, taking $S^1 = \reals/\integers = \Delta^1/\partial \Delta^1$ as our model for $S^1$. 

For $M$ any space, note that $S^1$ acts on $LM$, with the action map $\rho: S^1 \times LM \to LM$ given by $\rho(t, \gamma)(s) = \gamma(s + t)$. Then $\rho$ induces a map
\[
  H_p(S^1) \otimes H_q(LM) \xrightarrow{\times} H_{p + q}(S^1 \times LM) \xrightarrow{\rho_*} H_{p + q}(LM).
\]
For $\alpha \in H_p(LM)$, define $\Delta(\alpha) = \rho_*([S^1] \times \alpha)$, where $[S^1] \in H_1(S^1)$ is the fundamental class of $S^1$ determined by the quotient map $\Delta^1 \to \Delta^1/\partial\Delta^1$. Then $\Delta$ is a degree-$1$ operator on $H_*(LM)$. Since degree considerations force $\mu_*([S^1] \times [S^1]) \in H_2(S_1)$ to be $0$, $\Delta^2$ is identically $0$.

With a different choice $[S^1]'$ for the fundamental class, so that $[S^1]' = \lambda[S^1]$ for $\lambda \in k^\times$, then the corresponding operator $\Delta'$ is $\lambda\Delta$. In particular, choosing the opposite orientation for the cycle $\Delta^1 \to \Delta^1/\partial\Delta^1$, $t \mapsto 1 - t$, yields the operator $-\Delta$. 

We postpone detailed discussion of the Chas-Sullivan loop product on $H_*(LM)$ until Section~\ref{ssec:cs-loop-product}, where we give a homotopy-theoretic construction using Thom spectra due to Cohen and Jones~\cite{cohen-jones:2002}. For now, we record that the loop product arises from a combination of the degree-$(-d)$ intersection product on $H_*(M)$ and of the Pontryagin product on $H_*(\Omega M)$ induced by concatenation of based loops. Consequently, the loop product also exhibits a degree shift of $-d$:
\[
  \circ: H_p(LM) \otimes H_q(LM) \to H_{p + q - d}(LM).
\]
In order that $\circ$ define a graded algebra structure, we shift $H_*(LM)$ accordingly:

\begin{defn}
Denote $\Sigma^{-d} H_*(LM)$ as $\mathbb{H}_*(LM)$, called the \term{loop homology} of $M$, so that $\mathbb{H}_q(LM) = H_{q + d}(LM)$. 
\end{defn}

Under this degree shift, $\Delta$ gives a degree-$1$ operator on $\mathbb{H}_*(LM)$. One of the key results of Chas and Sullivan is that $\circ$ and $\Delta$ interact to give a BV algebra structure on $\mathbb{H}_*(LM)$. This BV algebra structure gives a canonical Gerstenhaber algebra structure~\cite[Prop~1.2]{getzler:1994}, and the resulting Lie bracket, denoted $\{-, -\}$, is called the \term{loop bracket}. The loop bracket can also be defined more directly using operations on Thom spectra~\cite{cohen-hess-voronov:2006,tamanoi:2007b}.

\subsubsection{Relations to Hochschild Constructions}

There are already substantial connections between the homology and cohomology of the free loop space $LX$ of a space $X$ and the Hochschild homology and cohomology of the DGAs $C_*\Omega X$ and $C^*X$. We state the key results that we employ below and survey the other relevant results.

The main result we will use is due to Goodwillie~\cite[\S V]{goodwillie:1985} and Burghelea and Fiedorowicz~\cite[Theorem~A]{burghelea-fiedorowicz:1986}.

\begin{thm}\label{thm:bfg-iso}
For $X$ a connected space, there is an isomorphism $BFG: HH_*(C_*\Omega X) \to H_*LM$ of graded $k$-modules, such that $BFG \circ B = \Delta \circ BFG$. 
\end{thm}

The proofs of this statement essentially rely on modeling the free loop space $LX$ as a cyclic bar construction on $\Omega X$, or a topological group replacement. Dually, Jones has shown that, for $X$ a simply connected space, $HH_*(C^*X) \isom H^*(LX)$, taking $B$ to a cohomological version of the $\Delta$ operator~\cite{jones:1987}. Jones's construction uses a cosimplicial model for $LX$, coming from the cyclic cobar construction on the space $X$ itself.  

In their homotopy-theoretic construction of string topology, Cohen and Jones modify this cyclic cobar construction to produce a cosimplicial model for the Thom spectrum $LM^{-TM}$ in terms of the manifold $M$ and the Atiyah dual $M^{-TM}$ of $M$~\cite{cohen-jones:2002}. Applying chains and Poincar\'e duality to this cosimplicial model yields an isomorphism
\[
  \mathbb{H}_*(LM) \isom HH^*(C^*M)
\]
of graded algebras, taking the loop product to the Hochschild cup product. As with Jones's earlier result, this isomorphism requires $M$ to be simply connected. When $k$ is a field of with $\fchar k = 0$, F\'elix and Thomas have identified a BV-algebra structure on $HH^*(C^*M)$ and have shown that it coincides with the string topology BV structure under this isomorphism~\cite{felix-thomas:2007}.

Koszul duality also provides a class of results relating the Hochschild cohomologies of different DGAs and hence providing other characterizations of string topology. In particular, F\'elix, Menichi, and Thomas have shown that, for $C$ a simply connected DGC with $H_*(C)$ of finite type, there is an isomorphism $HH^*(C^*) \isom HH^*(\Omega C)$ of Gerstenhaber algebras, where $C^*$ is the $k$-linear dual of $C$ and where $\Omega C$ is the cobar algebra of $C$~\cite{felix-menichi-thomas:2005}. When $C = C_*M$ for a simply connected manifold $M$, $C^* \isom C^*M$ and $\Omega C \simeq C_*\Omega M$, so $HH^*(C^*M) \isom HH^*(C_*\Omega M)$ as Gerstenhaber algebras. Combining this result with the isomorphism of Cohen and Jones gives an isomorphism of graded algebras $HH^*(C_*\Omega M) \isom \mathbb{H}_*(LM)$, again in the simply connected case.

Proceeding more directly from the $C_*\Omega M$ perspective above, Abbaspour, Cohen, and Gruher characterize the string topology of an aspherical $d$-manifold $M = K(G, 1)$ in terms of the group homology of the discrete group $G$~\cite{abbaspour-cohen-gruher:2008}. In particular, in this setting $G$ is a Poincar\'e duality group in the sense of group cohomology~\cite[{\S}VIII.10]{brown:1982}, and they establish a multiplication on the shifted group homology $H_{* + d}(G, kG^c)$, coming from a $G$-equivariant convolution product on $H^*(G; kG^c)$. They then show that this algebra is isomorphic to the graded algebra $\mathbb{H}_*(LM)$. By classical Hopf-algebra arguments, these group homology and cohomology groups are isomorphic to the Hochschild homology and cohomology of the group algebra $kG$. When $k$ is a field with $\fchar k = 0$, Vaintrob shows that $HH^*(kG)$ has a BV algebra structure and that this isomorphism is one of BV algebras~\cite{vaintrob:2007}.

Consequently, our main result Theorem~\ref{thm:main-st-hh} can be viewed as a generalization of these results to the case where $M$ is an arbitrary connected manifold and where $k$ is a general commutative ring for which $M$ is oriented.

%% file: sec-pd-rev.tex
\subsection{Derived Poincar\'e Duality}\label{ssec:pd}

\subsubsection{Derived Local Coefficients}

We will use these notions of $\Ext$ and $\Tor$ along with these combinatorial models to describe the string topology of $M$ in terms of modules over the algebra $C_*\Omega M$. Consequently, we discuss the development of homology and cohomology with these modules as local coefficients, focusing on a form of algebraic Poincar\'e duality for these modules. As mentioned above, this formulation has its roots in the notion of a Poincar\'e duality group, and our discussion more specifically originates in duality results of Klein~\cite{klein:2001} for topological groups and of Dwyer, Greenlees, and Iyengar~\cite{dwyer-greenlees-iyengar:2006} for connective ring spectra. 

We first generalize the notion of a local coefficient system on $X$. Suppose that $X$ is connected. Then by the Rothenberg-Steenrod constructions above, 
\[
  H_*(X) \isom H_*(B(k, C_*\Omega X, k)) = \Tor_*^{C_*\Omega X}(k, k).
\]
Moreover, the Borel construction $E(\Omega X) \times_{\Omega X} \pi_1 X$ provides a model for the universal cover $\tilde{X}$ of $X$, with the right action by $\pi_1X$. Consequently, $B(k, C_*\Omega X, k[\pi_1 X])$ provides a model for the right $k[\pi_1 X]$-module $C_*(\tilde{X}; k)$. Analogously, $B(k[\pi_1 X], C_*\Omega X, k)$ models $C_*(\tilde{X}; k)$ with the left $\pi_1 X$-action. 

Suppose that $E$ is a system of local coefficients in the usual sense, i.e., a right $k[\pi_1X]$-module. Under the morphism $C_*\Omega X \to H_0(\Omega X) = k[\pi_1 X]$ of DGAs, $E$ is a $C_*\Omega X$-mod\-ule, and
\begin{align*}
  C_*(X; E) &= E \otimes_{k[\pi_1 X]} C_*(\tilde{X}) \\
  & \simeq E \otimes_{k[\pi_1 X]} B(k[\pi_1 X], C_*\Omega X, k)
  \isom B(E, C_*\Omega X, k) \simeq E \otimes^L_{C_*\Omega X} k.
\end{align*}
Similarly, $C^*(X; E) \simeq R\Hom_{C_* \Omega X}(k, E)$, so passing to homology gives isomorphisms\[
  H_*(X; E) \isom \Tor_*^{C_*\Omega X}(E, k)
  \quad \text{and} \quad
  H^*(X; E) \isom \Ext^*_{C_*\Omega X}(k, E).
\]
Hence, $E \otimes^L_{C_*\Omega X} k$ and $R\Hom_{C_*\Omega X}(k, E)$ provide a generalization of homology and cohomology with local coefficients, where the coefficients are now $C_*\Omega X$-modules and where these theories take values in the derived category $\Ho Ch(k)$ of chain complexes over $k$.

\begin{defn}
For a $C_*\Omega X$-module $E$, let $H_\bullet(X; E) = E \otimes^L_{C_*\Omega X} k$, and let $H^\bullet(X; E) = R\Hom_{C_*\Omega X}(k, E)$. Let $H_*(X; E)$ and $H^*(X; E)$ denote their homologies.
\end{defn}

When $X$ is a Poincar\'e duality space, these ``derived'' versions of homology and cohomology satisfy a ``derived'' version of Poincar\'e duality:

\begin{thm}\label{thm:ext-tor-isom-from-pd}
Suppose $X$ is a $k$-oriented Poincar\'e duality space of dimension $d$. Let $z \in \Tor_d^{C_*\Omega X}(k, k)$ correspond to the fundamental class $[X] \in H_d(X)$. For $E$ a right $C_*\Omega M$-module, there is an evaluation map
\[
  \ev_{z, E}: H^\bullet(X; E) \to \Sigma^{-d} H_\bullet(X; E)
\]
that is a weak equivalence. On homology, this produces an isomorphism
\[
  H^*(X; E) \to H_{* + d}(X; E).
\]
When $E$ is a $k[\pi_1 X]$-module considered as a module over $C_*\Omega X$, this isomorphism coincides with the isomorphism coming from Poincar\'e duality for $X$ with local coefficients $E$. 
\end{thm}

This result directly generalizes the duality results for group cohomology, where cap product on a fundamental class $z \in \Tor_d^{kG}(k, k)$ induces an isomorphism $\Ext^i_{kG}(k, E) \to \Tor_{d - i}^{kG}(k, E)$. 

Klein's results for $G$ a topological group produce an equivalence in the reverse direction, which we summarize. In general, the spectrum $D_G = S[G]^{hG} = F(EG_+, S[G])^G$ has a norm map $D_G \wedge_{hG} E \to E^{hG}$ for $E$ a (naive) $G$-spectrum. When $BG$ is a finitely dominated $G$-complex of formal dimension $d$, this norm map is an equivalence, and $D_G$ is weakly equivalent to a negative sphere $S^{-d}$. Hence, this equivalence represents an equivalence between the derived functors of $S^0 \wedge_G -$ and $F_G(S^0, E)$ with a degree-$d$ shift.

\subsubsection{Duality for $A$-Modules}

Both Klein and Dwyer-Greenlees-Iyengar produce their duality results from basic Poincar\'e duality with local coefficients by five-lemma or triangle arguments together with finiteness properties. For example, many of the duality arguments in Klein come from the notion of a $G$-equivariant duality map~\cite{klein:1999b} $S^d \to Y \wedge_G Z$ for based $G$-complexes $Y, Z$. To detect a $G$-duality map, it suffices to have a $\pi$-equivariant duality map $S^d \to Y_{G_0} \wedge_{\pi} Z_{G_0}$, where $G_0$ is the connected component of the identity in $G$ and where $\pi = \pi_0(G) = G/G_0$.

We state an analogous detection theorem for $A$ a chain DGA (i.e., concentrated in nonnegative degrees), with $A = C_*G$ our primary example. As there are considerable simplifications to the proof in this case, we sketch its outline as well. We call an $A$-module \term{finite free} if it is built from $0$ by a finite number of pushouts of generating cofibrations, \term{finite} if it is a retract of a finite free module, and \term{homotopy finite} if it is weakly equivalent to a finite module. These homotopy finite modules are the same as the \term{small} modules in~\cite{dwyer-greenlees-iyengar:2006}, i.e., those modules in the thick subcategory generated by $A$ in the triangulated category $\modulecategory{A}$. 

An element $z \in (P \otimes_A Q)_n$ defines  evaluation maps $\ev_{z, M}: \Hom_A(P, M) \to M \otimes_A Q$ by $\ev_{z, M}(f) = (f \otimes \id_Q)(z)$. $\ev_{z, M}$ is a cycle if $z$ is, and hence is a chain map if $z$ is a $0$-cycle. On the derived level, then, $\alpha \in H_0(P \otimes^L_A Q)$ induces maps $\ev_{\alpha, M}: R\Hom_A(P, M) \to M \otimes_A^L Q$ well-defined up to homotopy. If $\ev_{\alpha, M}$ is a weak equivalence, it induces an isomorphism
\[
  \Ext^*_A(P, M) \to \Tor^A_*(M, Q).
\]
Finitely generated projective modules over an ordinary ring $R$ satisfy a strong form of duality~\cite[\S1.8]{brown:1982}, and as their analogues in $\modulecategory{A}$ so do the finite $A$-modules. Letting $M^* = \Hom_A(M, A)$ be the $A$-linear dual of $M$, the natural map $M \otimes_A M^* \to \Hom_A(M, M)$ is an isomorphism. Letting $z \in (M \otimes_A M^*)_0$ map to $\id_M$, $\ev_{z, N}$ is an equivalence for all $N$. Consequently, homotopy finite modules also possess such isomorphisms on the derived level.

Following Klein, we now wish to know a more basic criterion to determine whether a given class $\alpha$ as above is a dualizing class.

\begin{notation}
For $A$ a chain DGA, let $\tilde{A} = H_0A$, and let $\pi: A \to \tilde{A}$ be the surjective map taking $a \in A_0$ to the class $[a]$, and taking $a \in A_n$ to $0$ for $n > 0$.

For a left (resp., right) $A$-module $M$, let $\tilde{M}$ be the $\tilde{A}$-module $\pi^*\tilde{A} \otimes_A M$ (resp., $M \otimes_A \pi^*\tilde{A}$). Let $\pi_M: M \to \pi^*\tilde{M}$ be the induced surjective map of $A$-modules.
\end{notation}

We also recall that $H_*M$ is an $H_*A$-module, and so in particular each $H_jM$ is an ordinary $\tilde{A}$-module.

\begin{thm}\label{thm:dualizing-elt-detection}
Suppose $M, N$ are cofibrant, homotopy finite $A$-modules. Take $z \in H_0(M \otimes_A N)$, and let $z_0 = (\pi_M \otimes_\pi \pi_N)_*(z) \in H_0(\tilde{M} \otimes_{\tilde{A}} \tilde{N})$. Then $\ev_{z, E}$ is an equivalence for all $A$-modules $E$ if $\ev_{z_0, \tilde{A}}$ is an equivalence.
\begin{proof}
It suffices to reduce to the case when $M, N$ are finite: since $M, N$ are cofibrant and homotopy finite, there exist finite $A$-modules $F$ and $G$ homotopy equivalent to $M$ and $N$. Strong duality for finite modules then gives that, for every $\tilde{A}$-module $E$, $\ev_{z_0, E}$ is an equivalence, and this map is equivalent to
\[
  \ev_{z, \pi^*E}: \Hom_A(M, \pi^*E) \to \pi^*E \otimes_A N.
\]
Given a bounded below $A$-module $E$, the good truncations~\cite[p.~9]{weibel:1994} $P_jE$ give finite-length $A$-module quotients of $E$ with homology matching that of $E$ up to degree $j$. These quotients assemble into what is essentially an algebraic Postnikov tower for $E$, and the kernels $F_jE$ of the maps $P_jE \to P_{j - 1}E$ are equivalent to $\tilde{A}$-modules. Then five-lemma arguments and duality for $\tilde{A}$-modules show that each $\ev_{z, P_jE}$ is an equivalence, and the finiteness of $M$ and $N$ allows passage to an equivalence for $E$ as well.

Finally, any $A$-module $E$ is a colimit of bounded-below $A$-modules. Again by finiteness, $\Hom_A(M, -)$ and $- \otimes_A N$ commute with colimits, as does $H_*(-)$, so $\ev_{z, E}$ is an equivalence for all $E$. 
\end{proof}
\end{thm}

\subsubsection{Poincar\'e Duality for $C_*\Omega X$}

Suppose now that $X$ is a finite CW-complex satisfying Poincar\'e duality of formal dimension $d$ with respect to all $k[\pi_1 X]$-modules. Thus, for a given such module $E$, capping with a fundamental class $[X] \in H_d(X; k)$ induces isomorphisms
\[
  H^*(X; E) \to H_{* + d}(X; E)
\]
where, as usual, we view cohomology as being nonpositively graded. We reinterpret these properties in the context of duality for $C_*\Omega X$-modules.

First, since $X$ is a finite complex, the ground ring $k$ is homotopy finite~\cite[Prop.~5.3]{dwyer-greenlees-iyengar:2006}, and so its resolution $B(C_*\Omega X, C_*\Omega X, k)$ is cofibrant and homotopy finite. We take this module and the corresponding right-module resolution of $k$ as inputs to Theorem~\ref{thm:dualizing-elt-detection}. 

\begin{proof}[Proof of Theorem~\ref{thm:ext-tor-isom-from-pd}]
Throughout, let $A = C_*\Omega X$, so $\tilde{A} = k[\pi_1 X]$. Let 
\[
  M = B(k, A, A) 
  \quad \text{and} \quad 
  N = B(A, A, k).
\]
Since $C_*X$ is weakly equivalent to $B(k, A, k) \simeq M \otimes_A N \simeq \tilde{M} \otimes_{\tilde{A}} \tilde{N}$, let $z \in H_d(M \otimes_A N)$  correspond to a choice of fundamental class $[X] \in H_dM$. Let $z_0 = (\pi_M \otimes_\pi \pi_N)_*(z) \in H_d(\tilde{M} \otimes_{\tilde{A}} \tilde{N})$. Let $E$ be an $\tilde{A}$-module. Then $z_0$ induces
\[
  \ev_{z_0, E}: \Hom_{\tilde{A}}(\tilde{M}, E) \to E \otimes_{\tilde{A}} \Sigma^{-d} \tilde{N},
\]
which is a weak equivalence by Poincar\'e duality for $X$.

Applying Theorem~\ref{thm:dualizing-elt-detection}, 
\[
  \ev_{z, E}: \Hom_A(M, E) \to E \otimes_A \Sigma^{-d} N
\]
is a weak equivalence for all $E$, and it induces an isomorphism $\Ext_A^*(k, E) \to \Tor^A_{* + d}(E, k)$.

Rephrasing this in terms of ``derived'' homology and cohomology with local coefficients, cap product with $[X]$ induces a weak equivalence $\ev_{[X]}: H^\bullet(X; E) \to \Sigma^{-d} H_\bullet(X; E)$ and hence an isomorphism $H^*(X; E) \to H_{* + d}(X; E)$.
\end{proof}

%% file: ch-hh.tex
\section{Hochschild Homology and Cohomology}\label{ch:hh}

\input{sec-hh-pd.tex}

\input{sec-hh-cg.tex}

\input{sec-hh-adjoint.tex}

%% file: sec-hh-pd.tex
\subsection{Hochschild Homology and Poincar\'e Duality}\label{sec:hh-pd}

Now that we have established a Poin\-ca\-r\'e duality isomorphism $H^*(X; E) \to H_{* + d}(X; E)$ for $X$ a $k$-oriented Poincar\'e duality space of dimension $d$ and $E$ an arbitrary $C_*\Omega X$-module, we use it to construct an isomorphism between $HH^*(C_*\Omega X)$ and $HH_{* + d}(C_*\Omega X)$.

\begin{thm}\label{thm:hh-additive-isom}
For $X$ as above, derived Poincar\'e duality and the Hopf-algebraic properties of $C_*\Omega X$ produce a sequence of weak equivalences
\[
  \xymatrix{
	CH^*(C_*\Omega X) \ar[r]^-{\simeq} & H^{\bullet}(X; \Ad(\Omega X)) \ar[d]_{\simeq}^{\ev_{[X]}} \\
	CH_{* + d}(C_*\Omega X) \ar[r]^-{\simeq} & H_{\bullet + d}(X; \Ad(\Omega X))
}
\]
where $CH_*$ and $CH^*$ denote the Hochschild chains and cochains and where $\Ad(\Omega X)$ is $C_*\Omega X$ as a module over itself by conjugation, to be defined more precisely in Definition~\ref{defn:ad-module-top}. In homology, this gives an isomorphism $D: HH^*(C_*\Omega X) \to HH_{* + d}(C_*\Omega X)$ of graded $k$-modules. 
\end{thm}

\begin{corol}\label{corol:loop-homology-hh-isom}
The composite of $D$ and the Goodwillie isomorphism produces an additive isomorphism $HH^*(C_*\Omega X) \isom H_{* + d}(LX)$.
\end{corol}

In order to prove this result, we produce the horizontal isomorphisms relating the Hochschild homology and cohomology of $C_*\Omega X$ to $\Ext^*_{C_*\Omega X}$ and $\Tor_*^{C_*\Omega X}$. In fact, these equivalences hold for the algebra $C_*G$ of chains on a topological group $G$, and we develop them in that generality.

%% file: sec-hh-cg.tex
\subsection{Applications to \texorpdfstring{$C_*G$}{C*G}}\label{sec:hh-cg}

When $A$ is a Hopf algebra with strict antipode $S$, then the Hochschild homology and cohomology of $A$ can be expressed in terms of $\Ext^*_A(k, -)$ and $\Tor^A_*(-, k)$, using the isomorphisms of Prop.~\ref{prop:dgh-complex-isos}. In particular, let $M$ be an $A$-$A$-bimodule, considered canonically as a right $A^e$-module, and recall the adjoint DGA morphisms $\ad_0, \ad_1: A \to A^e$ from Definition~\ref{defn:adjoint-maps}. Then the Hochschild chains and cochains are isomorphic to
\begin{align*}
  CH_*(A, M) &= M \otimes_{A^e} B(A, A, A)
  \isom M \otimes_{A^e} B(\ad_0^* A^e, A, k)
  \isom B(\ad_0^* M, A, k), \\
  CH^*(A, M) &= \Hom_{A^e}(B(A, A, A), M)
  \isom \Hom_{A^e}(B(k, A, \ad_1^* A^e), M) \\
  & \isom \Hom_A(B(k, A, A), \Hom_{A^e}(\ad_1^* A^e, M))
  \isom \Hom_A(B(k, A, A), \ad_1^* M),
\end{align*}
which represent $\Tor^A_*(\ad_0^* M, k)$ and $\Ext_A^*(k, \ad_1^* M)$. These isomorphisms generalize the classical isomorphisms~\cite[Ex.~1.1.4]{loday:1992}
\[
  HH_*(kG, M) \isom H_*(G, \ad^*M) 
  \quad \text{and} \quad 
  HH^*(kG, M) \isom H^*(G, \ad^*M)
\]
and in fact reduce to them when $A$ is the group ring $kG$ of a discrete group $G$. (Since $kG$ is cocommutative, $\ad_0 = \ad_1$, and these pullback modules coincide.)

Suppose now that $G$ is a topological group, $A = C_*G$, and $S = C_*i$. Recall from Proposition~\ref{prop:inverse-almost-antipode} that $S^2 = \id$, that $S: A \to A^{\op}$ is a DGA isomorphism, but that $S$ satisfies the antipode identity for the DGH $C_*G$ only up to chain homotopy equivalence. Nevertheless, we show that we can relate the Hochschild homology and cohomology of $C_*G$ to $\Ext^*_{C_*G}(k, -)$ and $\Tor^{C_*G}_*(-, k)$. 

The adjoint $A$-module $\ad_0^*C_*G^e$ plays a key role in relating the Hochschild homology and cohomology of $C_*G$ to its $\Ext$ and $\Tor$ groups. In particular, since it is a pullback of an $A^e$-module, it behaves well with respect to both the formation of tensor products and $\Hom$-complexes, and hence with respect to the $\Hom$-$\otimes$ adjunction. As an intermediate result towards Theorem~\ref{thm:hh-additive-isom}, we will therefore establish the following homotopy equivalences of $A^e$-modules.

\begin{thm}\label{thm:adjoint-cg-he}
$B(\ad_0^* C_*G^e, C_*G, k)$ and $B(C_*G, C_*G, C_*G)$ are homotopy equivalent as left $C_*G^e$-modules, and $B(k, C_*G, \ad_0^* C_*G^e)$ and $B(C_*G, C_*G, C_*G)$ are homotopy equivalent as right $C_*G^e$-modules. 
\end{thm}

As a consequence of this theorem, we immediately obtain relations between Hochschild constructions and $\Ext$ and $\Tor$ over $C_*G$:

\begin{corol}\label{corol:cg-hh-isoms}
For $M \in \bimodulecategory{A}{A}$, considered canonically as a right $A^e$-module, there are homotopy equivalences
\begin{align*}
  & \Lambda_\bullet(G, M): M \otimes_{A^e} B(A, A, A) \xrightarrow{\simeq} B(\ad_0^*M, A, k), \\
  & \Lambda^\bullet(G, M): \Hom_{A^e}(B(A, A, A), M) \xrightarrow{\simeq} \Hom_A(B(k, A, A), \ad_0^*M).
\end{align*}
When $G$ and $M$ are understood, we omit them from the notation. Passing to homology, these induce isomorphisms
\begin{align*}
  &\Lambda_*(G, M): HH_*(A, M) \to \Tor^A_*(\ad_0^*M, k) 
  \quad \text{and} \\
  &\Lambda^*(G, M): HH^*(A, M) \to \Ext_A^*(k, \ad_0^*M).
\end{align*}
\begin{proof}
Since $\ad_0^*M \isom M \otimes_{A^e} \ad_0^* A^e$, the first equivalence 
of Theorem~\ref{thm:adjoint-cg-he} yields a homotopy equivalence $\Lambda_\bullet(G, M)$,
\[
  M \otimes_{A^e} B(A, A, A) \simeq M \otimes_{A^e} B(\ad_0^* A^e, A, k) \isom B(\ad_0^* M, A, k),
\]
which induces the isomorphism $\Lambda_*(G,M): HH_*(A, M) \to \Tor^A_*(\ad_0^*M, k)$. Likewise, since $\ad_0^* M \isom \Hom_{A^e}(\ad_0^* A^e, M)$, the second equivalence 
of Theorem~\ref{thm:adjoint-cg-he} yields a homotopy equivalence $\Lambda^\bullet(G, M)$,
\[
  \Hom_{A^e}(B(A, A, A), M) \simeq \Hom_{A^e}(B(k, A, \ad_0^* A^e), M) \isom \Hom_A(B(k, A, A), \ad_0^* M),
\]
which induces the isomorphism $\Lambda^*(G,M): HH^*(A, M) \to \Ext_A^*(k, \ad_0^*M)$.

When $G$ and $M$ are clear from context, we may drop them from the notation.
\end{proof}
\end{corol}

The $C_*G$-module $\ad_0^*M$ is actually not quite the definition we have in mind for the statement of Theorem~\ref{thm:hh-additive-isom}. We introduce this slightly more natural adjoint module:

\begin{defn}\label{defn:ad-module-top}
Let $G^{\op}$ be the group $G$ with the opposite multiplication.
Suppose that $X$ is a space with a left action $a_X$ by $G \times G^{\op}$. Pullback along $(\id \times i)\delta: G \to G \times G^{\op}$ makes $X$ a left $G$-space by ``conjugation,'' with $(g, x) \mapsto gxg^{-1}$. Let $\Ad_L(X)$ be $C_*X$ with the corresponding left $C_*G$-module structure.

Similarly, we may produce a right $C_*G$-module structure $\Ad_R(X)$ on $C_*X$. We denote these modules simply as $\Ad(X)$ when the module structure is clear from context.
\end{defn}

Note that these $\Ad(X)$ modules arise from first converting the $G \times G^{\op}$-action into a $G$-action and then applying $C_*$. Consequently, these modules arise more naturally in topological contexts, although they are not as immediately compatible with tensor product and $\Hom$-complex constructions. 

Let $K$ be another group. Note that if $X$ has a right $K$-action commuting with the left $G \times G^{\op}$-action, then $\Ad(X)$ is a $C_*G$-$C_*K$-bimodule. 

With the application of standard simplical techniques, the introduction of these $\Ad$ modules immediately provides a key intermediate step towards Theorem~\ref{thm:adjoint-cg-he}. First, we recall the two-sided bar construction in the topological setting.

\begin{defn}
Let $X$ be a left $G$-space and $Y$ a right $G$-space. Then the two-sided bar construction $B_\bullet(X, G, Y)$ is a simplicial space. Let $B(Y, G, X)$ be its geometric realization $|B_\bullet(Y, G, X)|$.
\end{defn}

\begin{prop} 
The maps $EZ$ induce a weak equivalence 
\[
  EZ: B(C_*Y, C_*G, C_*X) \to C_*(B(Y, G, X)).
\]
\begin{proof}
Take $n \geq 0$. Observe that $B_n(C_*Y, C_*G, C_*X) = C_*Y \otimes C_*G^{\otimes n} \otimes C_*X$, and that 
\[
  EZ: C_*Y \otimes C_*G^{\otimes n} \otimes C_*X \to C_*(Y \times G^n \times X) = C_*(B_n(Y, G, X))
\]
is a chain homotopy equivalence. Denote this map by $EZ_n$. By the form of the face and degeneracy maps $d_i$ and $s_i$ for $B_\bullet(Y, G, X)$ and for $B_\bullet(C_*Y, C_*G, C_*X)$, it follows that $C_*(d_i) EZ_n = EZ_{n - 1} d_i$ and $C_*(s_i) EZ_n = EZ_{n + 1} s_i$ for all $n$ and $0 \leq i \leq n$. Hence, the $EZ_*$ assemble to a chain homotopy equivalence
\[
  EZ: \Tot B_*(C_*Y, C_*G, C_*X) \to \Tot C_*(B_{\bullet}(Y, G, X)).
\]
Finally, since there is a weak equivalence from $\Tot C_*(E_\bullet)$ to $C_*(|E_\bullet|)$ for any simplicial space $E_\bullet$,~\cite[\S V.1]{goodwillie:1985}
the composite gives the desired weak equivalence.
\end{proof}
\end{prop}

Note also that if $Y$ or $X$ has an action by another group $H$ commuting with that of $G$, then the weak equivalence above is one of $C_*H$-modules.

\begin{corol}
In the notation above, taking $Y = *$ and $X = G \times G^{\op}$ with the standard left and right $G \times G^{\op}$ actions yields a weak equivalence of right $C_*(G \times G^{\op})$-modules
\[
  B(k, C_*G, \Ad(G \times G^{\op})) \to C_*(B(*, G, G \times G^{\op}))
\]
Taking $X = Y = G$ yields a weak equivalence of right $C_*(G)^e$-modules
\[
  B(C_*G, C_*G, C_*G) \to EZ^*C_*(B(G, G, G)).
\]
\end{corol}

We now relate the two topological bar constructions $B(*, G, G \times G^{\op})$ and $B(G, G, G)$.

\begin{prop}\label{prop:top-bar-group-homeo}
There are homeomorphisms $\phi^R: B(*, G, G \times G^{\op}) \leftrightarrows B(G, G, G): \gamma^R$ of right $G \times G^{\op}$-spaces and $\phi^L: B(G \times G^{\op}, G, *) \leftrightarrows B(G, G, G): \gamma^L$ of left $G \times G^{\op}$-spaces.
\begin{proof}
As $G$ is a Hopf-object with antipode in $\Top$ with its usual symmetric monoidal structure, these homeomorphisms come from the simplicial Hopf-object isomorphisms of Prop.~\ref{prop:dgh-complex-isos}. In particular, they define simplicial maps $\phi^R_\bullet: B_\bullet(*, G, G \times G^{\op}) \leftrightarrows B_\bullet(G, G, G): \gamma^R_\bullet$ by
\begin{align*}
  \phi^R_n(g_1, \dotsc, g_n, (g, g')) &= (g'(g_1 \dotsm g_n)^{-1}, g_1, \dotsc, g_n, g), \\
  \gamma^R_n(g', g_1, \dotsc, g_n, g) &= (g_1, \dotsc, g_n, (g, g' g_1 \dotsm g_n))
\end{align*}
Since $(g_1, \dotsc, g_n, (g, g')) \cdot (h, h') = (g_1, \dotsc, g_n, (gh, h'g'))$ and $(g', g_1, \dotsc, g_n, g) \cdot (h, h') = (h'g', g_1, \dotsc, g_n, gh)$, they are maps of right $G \times G^{\op}$-spaces. Applying geometric realization gives the $G \times G^{\op}$-equivariant homeomorphisms $\phi^R$ and $\gamma^R$.

Similarly, we obtain simplicial isomorphisms $\phi^L_\bullet: B_\bullet(G \times G^{\op}, G, *) \leftrightarrows B_\bullet(G, G, G): \gamma^L_\bullet$ by
\begin{align*}
  \phi^L_n((g, g'), g_1, \dotsc, g_n) &= (g, g_1, \dotsc, g_n, (g_1 \dotsm g_n)^{-1} g'), \\
  \gamma^L_n(g, g_1, \dotsc, g_n, g') &= ((g, g_1 \dotsm g_n g'), g_1, \dotsc, g_n)
\end{align*}
As above, these are isomorphisms simplical left $G \times G^{\op}$-spaces, and hence their geometric realizations give the $G \times G^{\op}$-equivariant homeomorphisms $\phi^R$ and $\gamma^R$. 
\end{proof}
\end{prop}

Combining these results, we obtain the following sequence of weak equivalences:

\begin{prop}\label{prop:adjoint-wk-equivs-top-modules}
There are weak equivalences
\[
\xymatrix{
	EZ^* B(k, C_*G, \Ad(G \times G^{\op})) \ar[d]^{EZ}_{\simeq}  & B(C_*G, C_*G, C_*G) \ar[d]^{EZ}_{\simeq} \\
	EZ^*C_*(B(*, G, G \times G^{\op})) \ar@<1ex>[r]^-{C_*(\phi)}_-{\isom} & EZ^*C_*(B(G, G, G)) \ar@<1ex>[l]^-{C_*(\gamma)}
	}
\]
of right $C_*G^e$-modules as indicated.
\end{prop}

%% file: sec-hh-adjoint.tex
\subsection{Comparison of Adjoint Module Structures}\label{sec:hh-adjoint-comp}

We now relate the $\Ad(X)$ $C_*G$-module structure to the $\ad_0^*$ pullback modules in Theorem~\ref{thm:adjoint-cg-he}. To do so, we employ the machinery of $A_\infty$-algebras, and in particular morphisms between modules over an $A_\infty$-algebra. Appendix~\ref{sec:a-infty-algs} contains the relevant details of this theory. 

We apply this theory to the adjoint modules discussed above. As before, let $X$ be a $(G \times G^{\op})$-space. Note that $C_*(G^{\op})$ and $(C_*G)^{\op}$ are isomorphic DGAs, and the homotopy equivalence $EZ_{G, G^{\op}}: C_*G \otimes C_*G^{\op} \to C_*(G \times G^{\op})$ is a morphism of DGAs. Since $C_*X$ is a left $C_*(G \times G^{\op})$-module, $EZ^*C_*X$ is a left $C_*G^e$-module, and so $\ad_0^*EZ^*C_*X = (EZ \ad_0)^* C_*X$ is another left $C_*G$-module structure on $C_*X$.

As will be shown below, these two $C_*G$-module structures on $C_*X$ factor through the left $C_*(G \times G)$-action $C_*(a_X) EZ_{G \times G, X} (C_*(\id \times i) \otimes \id)$. Hence, we consider such $C_*(G \times G)$-modules more generally. 

\begin{prop}
For $G, K$ groups and $A = C_*(G \times K)$, $\psi = EZ_{G, K} AW_{G, K}: A \to A$ is a DGA morphism.
\begin{proof}
Recall that the multiplication in $A$ is given by $\mu = C_*((m_G \times m_K) t_{(23)}) EZ_{G \times K, G \times K}$. We check that $\psi \mu = \mu (\psi \otimes \psi)$, using the $t_\sigma$ notation of \ref{sec:twist-notation}:
\begin{align*}
  \psi \mu 
  &= EZ_{G, K} AW_{G, K} C_*((m_G \times m_K) t_{(23)}) EZ_{G \times K, G \times K} \\
  &= C_*(m_G \times m_K) EZ_{G \times G, K \times K} AW_{G \times G, K \times K} C_*(t_{(23)}) EZ_{G \times K, G \times K} \\
  &= C_*(m_G \times m_K) EZ_{G \times G, K \times K} (EZ_{G,G} \otimes EZ_{K, K}) \tau_{(23)} (AW_{G, K} \otimes AW_{G, K}) \\
  &= C_*(m_G \times m_K) C_*(t_{(23)}) EZ_{G, K, G, K} (AW_{G, K} \otimes AW_{G, K}) \\
  &= C_*((m_G \times m_K) t_{(23)}) EZ_{G \times K, G \times K} (EZ_{G, K} AW_{G, K} \otimes EZ_{G, K} AW_{G, K}) = \mu (\psi \otimes \psi).
\end{align*}
Furthermore, since $EZ \circ AW = \id$ on $0$-chains, $\psi \eta = \eta$, so $\psi$ is a DGA morphism. 
\end{proof}
\end{prop}

In light of the interpretation given above of morphisms of $A_\infty$-modules between ordinary $A$-modules, the following proposition states that pullback along $EZ_{G,K} AW_{G,K}$ respects the action of $C_*(G \times K)$ only up to a system of higher homotopies.

\begin{prop}\label{prop:ez-aw-quasiiso}
Let $A = C_*(G \times K)$, and suppose that $M$ is a left $A$-module, with action $a_M$. Then $(EZ_{G,K} AW_{G,K})^* M$ is also a left $A$-module, which we denote $(L, a_L)$. There is a quasi-isomorphism $f: L \to M$ of $A_\infty$-modules over $A$, with $f_1: L \to M$ equal to $\id_M$. 
\begin{proof}
We construct the levels $f_n$ of this $A_\infty$-module morphism inductively using the theory of acyclic models~\cite[Ch.~13]{switzer:2002}. Let $H^0 = \id_k$ and let $H^1 = H$, the natural homotopy with $dH + Hd = EZ \circ AW - \id$. We then construct certain natural maps $H^n$ of degree $n$ of the form
\[
  H^n_{X_1, \dotsc, X_{2n}}: C_*(X_1 \times X_2) \otimes \dotsb \otimes C_*(X_{2n - 1} \times X_{2n}) \to C_*(X_1 \times \dotsb \times X_{2n})
\]
for spaces $X_1, \dotsc, X_{2n}$. Suppose that $H^n$ has been constructed. We abbreviate such expressions as $EZ_{X_1 \times X_2, X_3 \times X_4}$ to $EZ_{12, 34}$, for example.  With $1 \leq i \leq n$, define
\begin{align*}
  \hat{H}^{n + 1, 0} &= EZ_{12, 3 \dotsm (2n + 2)}(\id \otimes H^n), \\
  \hat{H}^{n + 1, i}
  &= C_*(t_{(2i \, 2i + 1)}) H^n (\id^{\otimes i - 1} \otimes (C_*(t_{(2 3)}) EZ_{(2i - 1)(2i), (2i + 1)(2i + 2)}) \otimes \id^{\otimes n - i}), \\
  \hat{H}^{n + 1, n + 1} &= EZ_{1 \dotsm (2n), (2n + 1) (2n + 2)}(H^n \otimes EZ_{2n + 1,2n + 2} AW_{2n + 1,2n + 2}).
\end{align*}
Essentially, the $\hat{H}^{n + 1, i}$ come from different ways of applying $H^n$ to $n + 1$ $C_*(X \times X')$ arguments in order. The end cases apply $H^n$ to the first or the last $n$ arguments, and then use $EZ$ to combine the $H^n$ output with the remaining $C_*(X \times X')$ factor. The middle cases instead combine two adjacent arguments into one singular chain, and then apply $H^n$.

Let $\hat{H}^{n + 1} = \sum_{i = 0}^{n + 1} (-1)^{i + 1} \hat{H}^{n + 1, i}$. A computation shows that $d\hat{H}^{n + 1} = \hat{H}^{n + 1}d$, so $\hat{H}^{n + 1}$ is a natural chain map of degree $n$. By the naturality of $\hat{H}^{n + 1}$, acyclic models methods apply to show that $\hat{H}^{n + 1} = d H^{n + 1} - (-1)^{n + 1} H^{n + 1} d$ for some natural map $H^{n + 1}$ of degree $n + 1$, as specified above. Since
\[
  d H^1_{1,2} + H^1_{1,2}d = EZ_{1,2} AW_{1,2} - \id
  = \hat{H}^{1,1}_{1,2} - \hat{H}^{1,0}_{1,2} = \hat{H}^1,
\]
the base case $n = 0$ also satisfies the property that $d H^{n + 1} - (-1)^{n + 1} H^{n + 1} d = \hat{H}^{n + 1}$. Consequently, such natural $H^n$ maps exist for all $n \geq 0$.

For $n \geq 1$, let $f_{n + 1} = a (C_*((m_G^{n - 1} \times m_K^{n - 1}) t_{\sigma_n}) H^n_{G, K, \dotsc, G, K} \otimes \id)$, where $\sigma_n \in S_{2n}$ takes $(1, \dotsc, 2n)$ to $(1, 3, \dotsc, 2n - 1, 2, \dotsc, 2n)$. By the construction of the $H^n$, these $f_n$ are seen to satisfy the conditions needed for the $A_\infty$-module morphism. 

Since $f_1 = \id$, which is a quasi-isomorphism of chain complexes, $f$ is a quasi-iso\-morph-ism of $A_\infty$-modules.
\end{proof}
\end{prop}

\begin{corol}\label{corol:adjoint-a-infty-quasi-isom}
There is a quasi-isomorphism $q: (EZ\ad_0)^*(C_*X) \to \Ad(X)$ of $A_\infty$-mod\-ules over $C_*G$.
\begin{proof}
Note that $a = C_*(a_X) EZ_{G \times G^{\op}, X} (C_*(\id \times i) \otimes \id)$ gives $C_*(X)$ a left $C_*(G \times G)$-module structure such that the module structure of $\Ad(X)$ is given by $a (C_*\delta \otimes \id)$. Furthermore, the $C_*G$-action $a'$ of $(EZ \ad_0)^*(C_*X)$ is given by
\begin{align*}
  a' &=  C_*(a_X) EZ_{G \times G^{\op}, X} (EZ_{G, G^{\op}} \otimes \id) (\id \otimes C_*i \otimes \id) (AW_{G, G} \otimes \id)( C_*\delta \otimes \id) \\
  &= C_*(a_X) EZ_{G \times G^{\op}, X} (C_*(\id \times i) \otimes \id) (EZ_{G, G} \otimes \id)  (AW_{G, G} \otimes \id)( C_*\delta \otimes \id) \\
  &= a ((EZ_{G,G} AW_{G,G} C_*\delta) \otimes \id).
\end{align*}
The proposition above then applies to the $C_*(G \times G)$-module structures $a$ and $(EZ \circ AW)^* a$ on $C_*X$ to yield an $A_\infty$ quasi-isomorphism. Pulling this morphism back along the DGA morphism $C_*\delta: C_*G \to C_*(G \times G)$ yields the desired quasi-isomorphism of $A_\infty$-modules over $C_*G$.
\end{proof}
\end{corol}

Consequently, this quasi-isomorphism of $A_\infty$-modules over $C_*G$ induces a quasi-iso\-morph\-ism of chain complexes $B(k, C_*G, (EZ\ad_0)^*(C_*X)) \to B(k, C_*G, \Ad(X))$. A similar argument shows that there exists a quasi-isomorphism $(EZ \ad_0)^* C_*X \to \Ad(X)$ of right $A_\infty$-modules for $X$ with a right $G \times G^{\op}$-action. Connecting these isomorphisms yields the following:

\begin{thm}
$B(k, C_*G, \ad_0^* C_*G^e)$ and $B(C_*G, C_*G, C_*G)$ are homotopy equivalent as right $C_*G^e$-modules.
\begin{proof}
Proposition~\ref{prop:adjoint-wk-equivs-top-modules},  Corollary~\ref{corol:adjoint-a-infty-quasi-isom}, and the quasi-isomorphism $EZ: C_*G^e \to C_*(G \times G^{\op})$ combine to produce the following diagram of weak equivalences and isomorphisms of right $C_*(G)^e$-modules:
\[
  \xymatrix{
	B(k, C_*G, \ad_0^* C_*G^e) \ar[d]_{\simeq}^{B(\id, \id, EZ)} & B(C_*G, C_*G, C_*G) \ar[ddd]^{EZ}_{\simeq} \\
	EZ^*B(k, C_*G, (EZ \ad_0)^* C_*(G \times G^{\op})) \ar[d]^{B(\id, \id, q)}_{\simeq} \\
	EZ^*B(k, C_*G, \Ad(G \times G^{\op})) \ar[d]^{EZ}_{\simeq} \\
	EZ^*C_*(B(*, G, G \times G^{\op})) \ar@<1ex>[r]^-{C_*(\phi)}_-{\isom} & EZ^*C_*(B(G, G, G)) \ar@<1ex>[l]^-{C_*(\gamma)}
	}
\]
Consequently, $B(k, C_*G, \ad_0^* C_*G^e)$ and $B(C_*G, C_*G, C_*G)$ are related by a zigzag of weak equivalences of $C_*G^e$-modules. Since they are both semifree, and hence cofibrant, $C_*G^e$-modules, Proposition~\ref{prop:cof-wk-equiv-htpy-equiv} implies that they are in fact homotopy equivalent.
\end{proof}
\end{thm}

We now complete the proof of Theorem~\ref{thm:hh-additive-isom}.

\begin{proof}[Theorem~\ref{thm:hh-additive-isom}]
By Corollary~\ref{corol:cg-hh-isoms}, Corollary~\ref{corol:adjoint-a-infty-quasi-isom}, and the naturality of this extended functoriality of $\Ext$ and $\Tor$ with respect to evaluation, we obtain the diagram
\[
\xymatrix{
	CH^{*}(C_*\Omega X) \ar[r]^-{\simeq} 
	& H^{\bullet}(X; \ad_0 C_*\Omega X^e)) \ar[d]^{\ev_{[X]}}_{\simeq} \ar[r]^{q_*}_{\simeq}
	& H^{\bullet}(X; \Ad(\Omega X)) \ar[d]^{\ev_{[X]}}_{\simeq} \\
	CH_{* + d}(C_*\Omega X) \ar[r]^-{\simeq} 
	& H_{\bullet + d}(X; \ad_0 C_*\Omega X^e)) \ar[r]^{q_*}_{\simeq}
	& H_{\bullet + d}(X; \Ad(\Omega X))
	}
\]
of weak equivalences. The outside of the diagram then provides the diagram of Theorem~\ref{thm:hh-additive-isom}.
\end{proof}

We note also that these techniques extend the multiplication map $\mu: C_*G \otimes C_*G \to C_*G$ to an $A_\infty$ map on $\Ad(G)$ that is compatible with the comultiplication on $C_*G$:

\begin{prop}\label{prop:a-inf-mult-ad}
There is a morphism of $A_\infty$-modules $\tilde{\mu}: \Delta^*(\Ad(G) \otimes \Ad(G)) \to \Ad(G)$ with $\tilde{\mu}_1 = \mu: C_*G \otimes C_*G  \to C_*G$.
\begin{proof}
Define $\tilde{\mu}$ as the following composite of $A$-module and $A_\infty$-$A$-module morphisms, where $f$ is the $A_\infty$-module morphism from Prop.~\ref{prop:ez-aw-quasiiso}:
\begin{align*}
  \Delta^*(\Ad(G) \otimes \Ad(G)) 
  &\xrightarrow{EZ} \Delta^* EZ^*(C_*(G^c \times G^c)) = C_*(\delta)^* (EZ \circ AW)^*(C_*(G^c \times G^c)) \\
  &\xrightarrow{f} C_*(\delta)^* (C_*(G^c \times G^c)) = C_*(\delta^*(G^c \times G^c)) \\
  &\xrightarrow{C_*(m)} C_*(G^c)
\end{align*}
Since $f_1 = \id$, $\tilde{\mu}_1 = C_*(m) \circ EZ = \mu$. 
\end{proof}
\end{prop}

We use this multiplication map in Chapter~\ref{ch:bv} to relate the $D$ isomorphism to a suitable notion of cap product in Hochschild homology and cohomology.

%% file: ch-bv.tex
\section{BV Algebra Structures}\label{ch:bv}

\input{sec-bv-loop-product.tex}

\input{sec-bv-lmtm-thh.tex}

\input{sec-bv-cap-pairings.tex}

\input{sec-bv-hh-bv-structure.tex}

%% file: sec-bv-loop-product.tex
\subsection{Multiplicative Structures}\label{sec:ring-structs}

As before, now let $M$ be a closed, connected, $k$-oriented manifold of dimension $d$. We now show that the isomorphism between $H_{* + d}(LM)$ and $HH^*(C_*\Omega M)$ established in Corollary~\ref{corol:loop-homology-hh-isom} is one of rings, taking the Chas-Sullivan loop product on $H_{*+d}(LM)$ to the Hochschild cup product on $HH^*(C_*\Omega M)$. To do so, we examine a homotopy-theoretic construction of the Chas-Sullivan product on the spectrum $LM^{-TM}$ and relate it to the ring spectrum structure of the topological Hochschild cohomology of the suspension spectrum $S[\Omega M]$. Again treating $\Omega M$ as a topological group $G$, we use the function spectrum $F_G(EG_+, S[G^c])$ as an intermediary, and we adapt some of the techniques of Abbaspour, Cohen, and Gruher~\cite{abbaspour-cohen-gruher:2008} and Cohen and Klein~\cite{cohen-klein:2009} to compare the ring spectrum structures. Smashing with the Eilenberg-Mac Lane spectrum and passing back to the derived category of chain complexes over $k$ then recovers our earlier chain-level equivalences. 

\subsection{Fiberwise Spectra and Atiyah Duality}\label{app:spectra}

We review some of the fundamental constructions and theorems in the theory of fiberwise spectra discussed in~\cite{cohen-klein:2009}. 

\begin{defn}
Let $X$ be a topological space, and let $\Top/X$ be the category of spaces over $X$. Let $\category{R}_X$ be the category of retractive spaces over $X$: such a space $Y$ has maps $s_Y: X \to Y$ and $r_Y: Y \to X$ such that $r_Y s_Y = \id_X$. Both categories are enriched over $\category{Top}$. Furthermore, there is a forgetful functor $u_X: \category{R}_X \to \Top/X$ with a left adjoint $v_X$ given by $v_X(Y) = Y \amalg X$, with $X \to Y \amalg X$ the inclusion. We often denote $v_X(Y)$ as $Y_+$ when $X$ is clear from context.

Given $Y \in \Top/X$, its \term{unreduced fiberwise suspension} $S_XY$ is the double mapping cylinder $X \union Y \times I \union X$; note this determines a functor $S_X: \Top/X \to \Top/X$. Given $Y \in \category{R}_X$, its \term{(reduced) fiberwise suspension} $\Sigma_X Y$ is $S_X Y \union_{S_X X} X$. This construction also determines a functor $\Sigma_X: \category{R}_X \to \category{R}_X$. 

Given $Y, Z \in \category{R}_X$, define their \term{fiberwise smash product} $Y \wedge_X Z$ as the pushout of $X \leftarrow Y \union_X Z \to Y \times_X Z$, where the map $Y \union_X Z \to Y \times_X Z$ takes $y$ to $(y, s_Z r_Y y)$ and $z$ to $(s_Y r_Z z, z)$. The fiberwise smash product then defines a functor $\wedge_X: \category{R}_X \times \category{R}_X \to \category{R}_X$ making $\category{R}_X$ a symmetric monoidal category, with unit $S^0 \times X$. Furthermore, $\Sigma_X Y \isom (S^1 \times X) \wedge_X Y$. 
\end{defn}

The notion of fiberwise reduced suspension is key in constructing spectra fibered over $X$.

\begin{defn}
A \term{fibered spectrum} $E$ over $X$ is a sequence of objects $E_j \in \category{R}_X$ for $j \in \naturals$ together with maps $\Sigma_X E_j \to E_{j + 1}$ in $\category{R}_X$.

Given $Y \in \category{R}_X$, its fiberwise suspension spectrum is the spectrum $\Sigma^\infty_X Y$ with $j$th space defined by $\Sigma_X^jY$, with the structure maps given by the identification $\Sigma_X(\Sigma_X^jY) \isom \Sigma_X^{j + 1}Y$.
\end{defn}

Spectra fibered over $X$ form a model category, with notions of weak equivalences, cofibrations, and fibrations arising as in the context of traditional spectra (i.e., spectra fibered over a point $*$). Given a spectrum $E$ fibered over $X$, one can produce covariant and contravariant functors from $\Top/X$ to spectra as follows.

\begin{defn}\label{defn:hom-cohom-fw-spectra}
Let $E$ be a spectrum fibered over $X$, and take $Y \in \Top/X$. Assume $E$ to be fibrant in the model structure of such fibered spectra. Define the spectrum $H_\bullet(Y; E)$ to be the homotopy cofiber of the map $Y \to Y \times_X E$. Define $H^\bullet(Y; E)$ levelwise to be the space $\Hom_{\Top/X}(Y^c, E_j)$ in level $j$, where $Y^c$ is a functorial cofibrant replacement for $Y$ in the category $\Top/X$.

Since both of these constructions are functorial in $Y$, they determine functors $H_\bullet(-; E)$ and $H^\bullet(-; E)$ which we call \term{homology} and \term{cohomology with $E$-coefficients}.
\end{defn}

We use these notions of spectrum-valued homology and cohomology functors to express a form of Poincar\'e duality. First, however, we must explain how to twist a fibered spectrum over $X$ by a vector bundle over $X$.

\begin{defn}
Let $E$ be a spectrum fibered over $X$ and let $\xi$ be a vector bundle over $X$. Define the \term{twist of $E$ by $\xi$}, ${}^\xi E$, levelwise by $({}^\xi E)_j = S^\xi \wedge_X E_j$, where $S^\xi$ is the sphere bundle over $X$ given by one-point compactification of the fibers of $\xi$. By introducing suspensions appropriately, the twist of $E$ by a virtual bundle $\xi$ is defined analogously.
\end{defn}

Poincar\'e or Atiyah duality can now be expressed in the following form, with coefficients in a spectrum $E$ over $N$:

\begin{thm}[\cite{cohen-klein:2009}]\label{thm:fw-atiyah-duality}
Let $N$ be a closed manifold of dimension $d$ with tangent bundle $TN$, and let $-TN$ denote the virtual bundle of dimension $-d$ representing the stable normal bundle of $N$. Let $E$ be a spectrum fibered over $N$. Then there is a weak equivalence of spectra
\begin{equation}
  H_\bullet(N; {}^{-TN} E) \simeq H^\bullet(N; E).
\end{equation}
Furthermore, this equivalence is natural in $E$.
\end{thm}

\subsubsection{The Chas-Sullivan Loop Product}\label{ssec:cs-loop-product}

We recall a homotopy-theoretic construction of the Chas-Sullivan loop product from \cite{cohen-jones:2002} in terms of umkehr maps on generalized Thom spectra, and we then illustrate how this loop product is expressed in \cite{cohen-klein:2009} using fiberwise spectra and fiberwise Atiyah duality. 

Let $M$ be a smooth, closed $d$-manifold, and note that $LM$ is a space over $M$ via the evaluation map at $1 \in S^1$, $\ev: LM \to M$. Let $L_\infty M$ be the space of maps of the figure-eight, $S^1 \vee S^1$, into $M$. Then
\[
  \xymatrix{
	L_\infty M \ar[r]^-{\tilde{\Delta}} \ar[d]^{\ev} & LM \times LM \ar[d]^{\ev \times \ev} \\
	M \ar[r]^-{\Delta} & M \times M
}
\]
is a pullback square. Furthermore, the basepoint-preserving pinch map $S^1 \to S^1 \vee S^1$ induces a map $\gamma$ of spaces over $M$:
\[
  \xymatrix{
	L_\infty M \ar[r]^-{\gamma} \ar[d]^{\ev} & LM\ar[d]^{\ev} \\
	M \ar@{=}[r] & M
  }
\]
Since the map $\tilde{\Delta}$ is the pullback of a finite-dimensional embedding of manifolds, it induces a collapse map
\[
  \Delta^!: (LM \times LM)_+ \to L_\infty M^{\nu_\Delta},
\]
where $\nu_\Delta$ here is the pullback along $\ev$ of the normal bundle $\nu_\Delta$ to the embedding $\Delta: M \to M \times M$. This normal bundle is isomorphic to $TM$, the tangent bundle to $M$. This collapse map is compatible with the formation of the Thom spectra of a stable vector bundle $\xi$ on $LM \times LM$. Taking $\xi = -TM \times -TM$, and noting that $\Delta^*(-TM \times -TM) = -TM \oplus -TM$, this gives an umkehr map
\[
  \Delta^!: (LM \times LM)^{-TM \times -TM} \to L_\infty M^{-TM}.
\]
Composing $\Delta^!$ with the smash product map $LM^{-TM} \wedge LM^{-TM} \to (LM \times LM)^{-TM \times -TM}$ and the map $\gamma^{-TM}: L_\infty M^{-TM} \to LM^{-TM}$ induced by $\gamma$ gives a homotopy-theoretic construction of the loop product
\[
  \circ: LM^{-TM} \wedge LM^{-TM} \to LM^{-TM}.
\]
A $k$-orientation of $M$ induces a Thom isomorphism $LM^{-TM} \wedge Hk \isom \Sigma^{-d} \Sigma^\infty LM_+ \wedge Hk$, so passing to spectrum homotopy groups gives the loop product on homology with the expected degree shift.

We now consider this loop product from the perspective of fiberwise spectra. Since $\ev: LM \to M$ makes $LM$ a space over $M$, $LM_+ = LM \amalg M$ is a retractive space over $M$, and iterated fiberwise suspensions of $LM$ over $M$ produce a fiberwise spectrum $\Sigma_M^\infty LM_+$ over $M$. Recall from Definition~\ref{defn:hom-cohom-fw-spectra} that for a spectrum $E$ fibered over $X$ and a space $Y$ over $X$, $H_\bullet(Y; E) = (Y \times_X E) \union CY$ and $H^\bullet(Y; E)$ is the spectrum of maps $\Map_X(Y, E^f)$ of $Y$ over $X$ into a fibrant replacement for $E$. 

\begin{prop}[\cite{cohen-klein:2009}]\label{prop:lmtm-spectrum-product}
As spectra, $LM^{-TM} \simeq H_\bullet(M; {}^{-TM}\Sigma_M^\infty LM_+)$. By fiberwise Atiyah duality,
\[
  LM^{-TM} \simeq H^\bullet(M; \Sigma_M^\infty LM_+).
\]
Since $LM$ is a fiberwise $A_\infty$-monoid over $M$, $\Sigma_M^\infty LM_+$ is a fiberwise $A_\infty$-ring spectrum, and so the spectrum of sections $H^\bullet(M; \Sigma_M^\infty LM_+)$ is also a ring spectrum. The Chas-Sullivan loop product on $LM^{-TM}$ arises as the induced product on $LM^{-TM}$. 
\begin{proof}
We check that $LM^{-TM} \simeq H_\bullet(M; {}^{-TM}\Sigma_M^\infty LM_+)$. Let $\nu$ be an $(L - d)$-dimensional normal bundle for $M$. The $(j + L)$th space of $LM^{-TM}$ is then the Thom space $LM^{\ev^*(\nu) \oplus \epsilon^j}$, the one-point compactification of $\ev^*(\nu) \oplus \epsilon^j$. The $(j + L)$th space of ${}^{-TM}\Sigma_M^\infty LM_+$ is $S^j \wedge S^\nu \wedge_M LM_+$, which is seen to be the fiberwise compactification of $LM^{\ev^*(\nu) \oplus \epsilon^j}$ over $M$. Applying $H_\bullet(M; -)$ attaches the cone $CM$ to this space along the $M$-section of basepoints added by the fiberwise compactification, thus making a space homotopy equivalent to $LM^{\ev^*(\nu) \oplus \epsilon^j}$.

Fiberwise Atiyah duality then shows that $LM^{-TM} \simeq H^\bullet(M; \Sigma_M^\infty LM_+)$.

We compare each step of the original $LM^{-TM}$ construction of the loop product to the ring spectrum structure on $H^\bullet(M; \Sigma_M^\infty LM_+)$. First, since $\Sigma_M^\infty LM_+ \wedge \Sigma_M^\infty LM_+$ is naturally isomorphic to $\Sigma_{M \times M}^\infty L(M \times M)_+$ as spectra fibered over $M \times M$, the square
\[
  \xymatrix{
	LM^{-TM} \wedge LM^{-TM} \ar[r]^-{\wedge} \ar[d]^{\simeq} & 
	(L(M \times M))^{-T(M \times M)} \ar[d]^{\simeq} \\
	\Gamma(\Sigma_M^\infty LM_+) \wedge \Gamma(\Sigma_M^\infty LM_+) \ar[r]^-{\wedge} & 
	\Gamma(\Sigma_{M \times M}^\infty L(M \times M)_+)
}
\]
commutes. Next, pullback along $\Delta$ induces a map of spectra
\[
  \Delta^\bullet: H^\bullet(M \times M, \Sigma^\infty_{M \times M}L(M \times M)_+) \to H^\bullet(M, \Sigma^\infty_{M \times M}L(M \times M)_+).
\]
The universal property of the pullback $L_\infty M$ induces a homeomorphism between the spaces $\Map_{M \times M}(M, L(M \times M)_+)$ and $\Map_M(M, L_\infty M_+)$, and thus an equivalence 
\[
 H^\bullet(M, \Sigma^\infty_{M \times M}L(M \times M)_+) \simeq H^\bullet(M, \Sigma^\infty_M L_\infty M_+).
\]
Hence, the umkehr map diagram
\[
  \xymatrix{
	(LM \times LM)^{-T(M \times M)} \ar[r]^-{\Delta^!} \ar[d]^{\simeq} &
	L_\infty M^{-TM} \ar[d]^{\simeq} \\
	H^\bullet(M, \Sigma_{M \times M}^\infty L(M \times M)_+) \ar[r]^-{\Delta^\bullet} &
	H^\bullet(M, \Sigma^\infty_M L_\infty M_+)
}
\]
commutes. Finally, by the naturality of fiberwise Atiyah duality in the spectrum argument, the diagram
\[
  \xymatrix{
	L_\infty M^{-TM} \ar[r]^{\gamma^{-TM}} \ar[d]^{\simeq} &
	LM^{-TM} \ar[d]^{\simeq} \\
	H^{\bullet}(M, \Sigma^\infty_M L_\infty M_+) \ar[r]^{\gamma^{\bullet}} &
	H^{\bullet}(M, \Sigma^\infty_M LM_+)
}
\]
commutes.
\end{proof}
\end{prop}

%% file: sec-bv-lmtm-thh.tex
\subsubsection{Ring Spectrum Equivalences}\label{ssec:ring-spectrum-equivs}

For notational simplicity, let $G$ be a topological group replacement for $\Omega X$. Furthermore, if $Y$ is an unbased space, we follow Klein~\cite{klein:2001} in letting $S[Y]$ denote the fibrant replacement of the suspension spectrum of $Y_+$. Thus, the $j$th space of $S[Y]$ is $Q(S^j \wedge Y_+)$, where $Q = \Omega^\infty\Sigma^\infty$ is the stable homotopy functor. Furthermore, we let $E^f$ denote a fibrant replacement for a spectrum $E$ fibered over a space $Z$; if the fibers are suspension spectra $\Sigma^\infty Y_+$, then the fibers of $E^f$ may be taken to be $S[Y]$. 

We establish spectrum-level analogues of the Goodwillie isomorphism $BFG$ and the isomorphism $\Lambda_*(G, M)$.

\begin{prop}
There are equivalences of spectra $\Gamma: S[LM] \to S[G] \wedge_G EG_+$ and $\Lambda_\bullet: S[G] \wedge_G EG_+ \to THH^S(S[G])$.
\begin{proof}
We first establish the equivalence $\Gamma$. Since $G \simeq \Omega M$, $M \simeq BG$. Furthermore, $LM \simeq LBG$ over this equivalence, and so $\Sigma^\infty LM_+ \simeq \Sigma^\infty LBG_+$. Next, the well-known homotopy equivalence $LBG \simeq G^c \times_G EG$ shows that 
\[
  \Sigma^\infty LBG_+ \simeq \Sigma^\infty(G^c \times_G EG)_+.
\]
Passing to fibrant replacements then gives $S[LM] \simeq S[G^c] \wedge_G EG_+$.

Take $B(G, G, *) = |B_\bullet(G, G, *)|$ as a model for $EG$. Then 
\[
  W_n = S[G^c] \wedge_G B_n(G, G, *)_+
\]
determines a simplicial spectrum with $|W_\bullet| = S[G^c] \wedge_G EG_+$. Likewise, 
\[
  V_n = S[G] \wedge_{G \times G^{\op}} B_n(G, G, G)_+
\]
determines a simplicial spectrum such that $|V_\bullet| \simeq THH^S(S[G])$. Hence, we show there is an isomorphism $\chi_\bullet: W_\bullet \xrightarrow{\isom} V_\bullet$ of simplicial spectra. In fact, this map is the composite of the isomorphism
\[
  S[G^c] \wedge_G B_n(G, G, *)_+ \isom S[G] \wedge_{G \times G^{\op}} B_n(G \times G^{\op}, G, *)_+
\]
and $S[G] \wedge_{G \times G^{\op}} \phi^L_\bullet$, where $\phi^L_\bullet$ is the simplicial homeomorphism $B_\bullet(G \times G^{\op}, G, *) \to B_\bullet(G, G, G)$ of Proposition~\ref{prop:top-bar-group-homeo}. Explicitly, the $\chi_n$ are given by
\[
  \chi_n(a \wedge [g_1 \mid \dotsb \mid g_n]) = (g_1 \dotsm g_n) a \wedge [g_1 \mid \dotsb \mid g_n].
\]
\end{proof}
\end{prop}

We also produce spectrum-level analogues of the weak equivalences among $C_{* + d}(LM)$, $R\Hom_{C_*G}^*(k, \Ad(G))$, and $CH^*(C_*G)$. Westerland has shown~\cite{westerland:2006} that $F_G(EG_+, \Sigma^\infty G^c_+)$ is a ring spectrum for $G$ a general topological group, and the topological Hochschild cohomology $THH_S(S[G])$ of the ring spectrum $S[G]$ is likewise well-known to be a ring spectrum itself. The composite isomorphism should be equivalent to Klein's~\cite{klein:2006} equivalence of spectra $(LX)^{-\tau_X} \simeq THH_S(S[\Omega X])$ for a Poincar\'e duality space $(X , \tau_X)$, although we make more of the ring structure explicit here.

We first relate $LM^{-TM}$ and $F_G(EG_+, S[G^c])$ as ring spectra.

\begin{prop}
There is an equivalence of ring spectra $\Psi: LM^{-TM} \to F_G(EG_+, S[G^c])$.
\begin{proof}
Recall from Proposition~\ref{prop:lmtm-spectrum-product} that the spectrum 
\[
  LM^{-TM} \simeq H_\bullet(M; {}^{-TM}\Sigma^\infty_M LM) \simeq H^\bullet(M; \Sigma^\infty_M LM).
\]
Since $LM \simeq LBG$ over the equivalence $M \simeq BG$, $H^\bullet(M; \Sigma^\infty_M LM)$ and $H^\bullet(BG; \Sigma^\infty_{BG} LBG)$ are equivalent spectra. Since $LBG \simeq G^c \times_G EG$, this is equivalent to $H^\bullet(BG; \Sigma^\infty_{BG} (G^c \times_G EG)_+)$. The $j$th level of the target spectrum $\Sigma^\infty_{BG} (G^c \times_G EG)_+$ is the space 
\[
  (S^j \times BG) \wedge_{BG} ((G^c \times_G EG) \amalg BG) \isom (S^j \wedge G^c_+) \times_G EG.
\]
Since the $H^\bullet$ construction implicitly performs a fibrant replacement on its target, the $j$th space of the spectrum $H^\bullet(BG; \Sigma^\infty_{BG} (G^c \times_G EG)_+)$ is $\Map_{BG}(BG, Q(S^j \wedge G^c_+) \times_G EG)$. 

For a $G$-space $Y$, to pass from $\Map_{BG}(BG, Y \times_G EG)$ to $\Map_G(EG, Y)$, we form the pullback diagram
\[
  \xymatrix{
	& Y \ar@{=}[r] \ar[d] & Y \ar[d] \\
	G \ar@{=}[d] \ar[r] & (Y \times_G EG) \times_{BG} EG \ar[r] \ar[d] & Y \times_G EG \ar[d] \\
	G \ar[r] & EG \ar[r]^{/G} \ar@/_12pt/[u]_{\tilde{\sigma}} & BG \ar@/_12pt/[u]_{\sigma}
}
\]
Note that $(Y \times_G EG) \times_{BG} EG$ has a left $G$-action coming from the right $EG$ factor. By pullback, $\sigma \in \Map_{BG}(BG, Y \times_G EG)$ determines a section $\tilde{\sigma} \in \Map_{EG}(EG, (Y \times_G EG) \times_{BG} EG)$, with $\tilde{\sigma}(e) = (\sigma([e]), e)$. Since $\tilde{\sigma}(ge) = (\sigma([ge]), ge) = g \cdot (\sigma([e]), e)$, $\tilde{\sigma}$ is $G$-equi\-var\-i\-ant.

Since $EG$ is a free $G$-space, $(Y \times_G EG) \times_{BG} EG$ is homeomorphic to $Y \times EG$ by the map taking $([y, e], e)$ to $(y, e)$. Furthermore, $Y \times EG$ has a left $G$-action given by $\Delta_G^*(i^*Y \times EG)$, where the pullback $i^*$ by the inverse map $i$ for $G$ converts the right $G$-space $Y$ into a left $G$-space. Since
\[
  g \cdot ([y,e], e) = ([yg^{-1}, ge], ge) \mapsto (yg^{-1}, ge) = g\cdot (y,e), 
\]
this homeomorphism is $G$-equivariant with the above left $G$-action on $(Y \times_G EG) \times_{BG} EG$. Hence, $\tilde{\sigma}$ corresponds to a $G$-equivariant section $\sigma' \in \Map_{EG}(EG, Y \times EG)$. Since $Y \times EG$ is a product, $\sigma'$ is determines by the projections $\pi_{EG} \circ \sigma' = \id_{EG}$ and $\pi_Y \circ \sigma'$, both of which are $G$-equivariant maps. Consequently, we obtain the homeomorphism $\Map_{BG}(BG, Y \times_G EG) \isom \Map_G(EG, Y)$.  

Applying this correspondence levelwise with $Y = Q(S^j \wedge G^c_+)$, this space of sections is homeomorphic to $\Map_G(EG, Q(S^j \wedge G^c_+))$, the $j$th space of $F_G(EG_+, S[G^c])$.

We now show that under these equivalences, the product on $LM^{-TM}$ coincides with that on $F_G(EG_+, S[G^c])$. 
From above, the product on $LM^{-TM}$ is equivalent to that on $H^\bullet(M; \Sigma^\infty_M LM_+)$, given by
\begin{multline*}
  H^\bullet(M; \Sigma^\infty_M LM_+)^{\wedge 2} 
  \xrightarrow{\wedge} H^\bullet(M \times M; \Sigma^\infty_{M \times M} (LM \times LM)_+) \\
  \xrightarrow{\Delta^\bullet_M} H^\bullet(M; \Sigma^\infty_M(L_\infty M)_+)
  \xrightarrow{\gamma_*} H^\bullet(M; \Sigma^\infty_M (LM)_+).
\end{multline*}
Since $LM \simeq LBG$ over $M \simeq BG$, and since $LBG \simeq G^c \times_G EG$ as fiberwise monoids over $BG$~\cite[App.~A]{gruher:2007b}, this sequence is equivalent to
\begin{multline*}
  H^\bullet(BG; \Sigma^\infty_{BG}(G^c \times_G EG)_+)^{\wedge 2}
  \xrightarrow{\wedge} H^\bullet(B(G \times G), \Sigma^\infty_{B(G \times G)}((G^c \times G^c) \times_{G \times G} E(G \times G))_+) \\
  \xrightarrow{B(\Delta_G)^\bullet} H^\bullet(BG, \Sigma^\infty_{BG}(\Delta^*_G(G^c \times G^c) \times_G EG)_+)
  \xrightarrow{\mu_*} H^\bullet(BG, \Sigma^\infty_{BG}(G^c \times_G EG)_+)
\end{multline*}
Finally, passing to equivariant maps into the fibers, this sequence is equivalent to
\begin{multline*}
  F_G(EG_+, S[G^c])^{\wedge 2}
  \xrightarrow{\wedge} F_{G \times G}(E(G \times G)_+, S[G^c \times G^c]) \\
  \xrightarrow{E(\Delta_G)^* \circ \Delta_G^*} F_G(EG_+, S[\Delta^*_G(G^c \times G^c)]))
  \xrightarrow{S[\mu]_*} F_G(EG_+, S[G^c]).
\end{multline*}
This is the description given by Westerland~\cite{westerland:2006} of the ring structure of $F_G(EG_+, S[G^c])$. 
\end{proof}
\end{prop}

We now relate $F_G(EG_+, S[G^c])$ and $THH_S(S[G])$ as ring spectra.

\begin{prop}
There is an equivalence $\Lambda^\bullet(G): F_G(EG_+, S[G^c]) \to THH_S(S[G])$ of ring spectra.
\begin{proof}
We first show that $F_G(EG_+, S[G^c]) \simeq THH_S(S[G])$ as spectra. As above, take $B(G, G, *)$ as a model of $EG$ with the usual left $G$-action. Let 
\[
  Z^n = \Map_G(B_n(G, G, *), S[G^c])
\]
be the corresponding cosimplicial spectrum; then $F_G(EG_+, S[G^c]) = \Tot Z^\bullet$. Similarly, the endomorphism operad $F_S(S[G]^{\wedge\bullet}, S[G])$ of $S[G]$ is an operad with multiplication on account of the unit and multiplication maps of $S[G]$, and so by results of McClure and Smith~\cite{mcclure-smith:2002}, it admits a canonical cosimplicial structure. Furthermore, its totalization is $THH_S(S[G])$. Let 
\[
  Y^n = \Map(G^n, S[G]),
\]
which is equivalent to $F_S(S[G]^{\wedge n}, S[G])$. Then $\Tot Y^\bullet \simeq THH_S(S[G])$. Consequently, we need only exhibit an isomorphism $\psi^\bullet: Z^\bullet \to Y^\bullet$ of cosimplicial spectra. 

Since $G^c = \ad^*G$, the equivariance adjunction between $G$-spaces and $G \times G^{\op}$-spaces followed by pullback along the simplicial homeomorphism $\gamma^L_\bullet$ of Proposition~\ref{prop:top-bar-group-homeo} gives that
\[
  \Map_G(B_\bullet(G, G, *), S[G^c])
  \isom \Map_{G \times G^{\op}}(B_\bullet(G, G, G), S[G])
\]
as cosimplicial spectra. Finally, by the freeness of each $B_q(G, G, G)$ as a $(G \times G^{\op})$-space, 
\[
  \Map_{G \times G^{\op}}(B_\bullet(G, G, G), S[G]) \isom \Map(G^{\bullet}, S[G]),
\]
also as cosimplicial spectra. Explicitly, for $a \in Z^p$, we have that
\[
  \psi^p(a)([g_1 \mid \dotsb \mid g_p]) =  a([g_1 \mid \dotsb \mid g_p]) (g_1 \dotsm g_p).
\] 

We now show this equivalence is one of ring spectra. Since the cosimplicial structure of $Y^\bullet$ comes from an operad with multiplication, it has a canonical cup-pairing $Y^p \wedge Y^q \to Y^{p + q}$ coming from the operad composition maps and the multiplication. Therefore, $\Tot Y^\bullet \simeq THH_S(S[G])$ is an algebra for an operad $\mathcal{C}$ weakly equivalent to the little $2$-cubes operad. Hence, $THH_S(S[G])$ is an $E_2$-ring spectrum, and \emph{a fortiori} an $A_\infty$-ring spectrum.

We can also describe the product on $F_G(EG_+, S[G^c])$ cosimplicially. The diagonal map $\Delta_G$ applied levelwise gives a canonical $G$-equivariant diagonal map $\Delta$ on the simplicial set $B_\bullet(G, G, *)$. On realizations, this maps gives the $G$-equivariant map $E(\Delta_G): EG \to EG \times EG$. This diagonal map induces a sequence of cosimplicial maps
\begin{multline*}
  \Map_G(B_q(G, G, *), S[G^c]) \wedge \Map_G(B_q(G, G, *), S[G^c]) \\
  \to \Map_{G \times G}(B_q(G, G, *) \times B_q(G, G, *), S[G^c \times G^c])
  \to \Map_G(B_q(G, G, *), S[G^c])
\end{multline*}
when composed with pullback along $\Delta_G$ and with $S[\mu]$. These assemble to a cosimplicial multiplication map $\Tot Z^\bullet \wedge \Tot Z^\bullet \to \Tot Z^\bullet$, which produces the strictly associative product map on $F_G(EG_+, S[G^c])$.

This cosimplicial map induces a canonical cup-pairing $\cup: Z^p \wedge Z^q \to Z^{p + q}$, and as a result there is a map $\Tot Z^\bullet \wedge \Tot Z^\bullet \to \Tot Z^\bullet$ for each $u$ with $0 < u < 1$. Each such map is also  homotopic to the strict multiplication on $\Tot Z^\bullet$. Furthermore, these maps assemble into an action of the little $1$-cubes operad on $\Tot Z^\bullet$, making it an $A_\infty$-ring spectrum.

Consequently, we need to show that the isomorphism $\psi^\bullet: Z^\bullet \to Y^\bullet$ of cosimplicial spectra induces an isomorphism of these two cup-pairings. We do that explicitly using the definition of $\psi^\bullet$ and these cup-pairings. If $a \in Z^p$ and $b \in Z^q$, then
\begin{align*}
  (a \cup b)([g_1 \mid \dotsb \mid g_{p + q}]
  &= a([g_1 \mid \dotsb \mid g_p]) b(g_1 \dotsm g_p [g_{p + 1} \mid \dotsb \mid g_{p + q}]) \\
  &= a([g_1 \mid \dotsb \mid g_p]) g_1 \dotsm g_p b([g_{p + 1} \mid \dotsb \mid g_{p + q}]) (g_1 \dotsm g_p)^{-1}.
\end{align*}
Hence,
\begin{align*}
		(\psi^p(a) \cup \psi^q(b))([g_1 \mid \dotsb& \mid g_{p + q}]) \\
	&= \psi(a)([g_1 \mid \dotsb \mid g_p]) \psi(b)([g_{p + 1} \mid \dotsb \mid g_{p + q}]) \\
	&=  a([g_1 \mid \dotsb \mid g_p]) (g_1 \dotsm g_p) b([g_{p + 1} \mid \dotsb \mid g_{p + q}]) (g_{p + 1} \dotsm g_{p + q}) \\
	&=  (a \cup b)([g_1 \mid \dotsm \mid g_{p + q}]) (g_1 \dotsm g_{p + q})\\
	&= \psi^{p + q}(a \cup b)([g_1 \mid \dotsm \mid g_{p + q}]).
\end{align*}
Therefore, $\psi^p(a) \cup \psi^q(b) = \psi^{p + q}(a \cup b)$, so $\psi$ induces an isomorphism of cup-pairings, as desired. We conclude that the $A_\infty$-ring structures on $F_G(EG_+, S[G^c])$ and $THH_S(S[G])$ are equivalent.
\end{proof}
\end{prop}

Naturally, we want to connect these equivalences of spectra to the quasi-isomorphisms of $k$-chain complexes determined above. Smashing these spectra with $Hk$, the Eilenberg-Mac Lane spectrum of $k$, and using the smallness of $EG$ as a $G$-space when $M$ is a Poincar\'e duality space,
\begin{align*}
  F_G(EG_+, S[G^c]) \wedge Hk &\simeq F_G(EG_+, S[G^c] \wedge Hk)
  \simeq F_{G_+ \wedge Hk}(EG_+ \wedge Hk, S[G^c] \wedge Hk)
\end{align*}
and $THH_S(S[G]) \simeq THH_{Hk}(S[G] \wedge Hk)$. Thus,
\[
  LM^{-TM} \wedge Hk \simeq F_{S[G] \wedge Hk}(EG_+ \wedge Hk, S[G^c] \wedge Hk) 
\simeq THH_{Hk}(S[G] \wedge Hk).
\] 
Furthermore, since $M$ is $k$-oriented, $LM^{-TM} \wedge Hk \simeq \Sigma^{-d} S[LM]$ by the Thom isomorphism.

We relate these spectrum-level constructions back to the chain-complex picture above. By results of Shipley~\cite{shipley:2007}, there is a zigzag of Quillen equivalences between the model categories of $Hk$-algebras and DGAs over $k$; the derived functors between the homotopy categories are denoted $\Theta: Hk\text{-alg} \to \text{DGA}/k$ and $\mathbb{H}: \text{DGA}/k \to Hk\text{-alg}$. Furthermore, this correspondence induces Quillen equivalences between $\modulecategory{A}$ and $\modulecategory{\mathbb{H}A}$ for $A$ a DGA over $k$, and between $\modulecategory{B}$ and $\modulecategory{\Theta B}$ for $B$ an $Hk$-algebra.

We also have that $\Theta(S[G] \wedge Hk)$ is weakly equivalent to $C_*(G; k)$, and so their categories of modules are also Quillen equivalent, since $\Ch(k)$ exhibits Quillen invariance for modules~\cite[3.11]{schwede-shipley:2003a}. Consequently, the categories of modules over $S[G] \wedge Hk$ and over $C_*(G; k)$ are Quillen equivalent. Since $EG_+ \wedge Hk$ is equivalent to $C_*(EG; k) \simeq k$, the equivalence above gives the quasi-isomorphisms
\[
  C_{* + d}(LM) \simeq R\Hom_{C_*\Omega M}(k, \Ad(\Omega M)) \simeq CH^*(C_*\Omega M)
\]
we developed above. In paticular, we recover the derived Poincar\'e duality map as the composite of the Atiyah duality map and the Thom isomorphism. Applying $H_*$, we recover the isomorphisms
\[
  H_{* + d}(LM) \simeq \Ext^*_{C_*\Omega M}(k, \Ad(\Omega M)) \simeq HH^*(C_*\Omega M).
\]
We summarize these results in the following theorem:

\begin{thm}
Any model for $R\Hom_{C_*\Omega M}(k, \Ad(\Omega M))$ is an algebra up to homotopy (i.e., a monoid in $\Ho \Ch(k)$) coming from the ring spectrum structure of $F_{\Omega M}(E\Omega M_+, S[\Omega M^c])$. Furthermore, this algebra is equivalent to the $A_\infty$-algebra $CH^*(C_*\Omega M)$, and in homology induces the loop product on $H_{* + d}(LM)$.
	
Therefore, the isomorphism $BFG \circ D: HH^*(C_*\Omega M) \to H_{* + d}(LM)$ is one of graded algebras, taking the Chas-Sullivan loop product to the Hochschild cup product. 
\end{thm}

In particular, the model $\Hom_A(B(k, A, A), \Ad(G))$ for $A = C_*G$ has an $A_\infty$-algebra structure arising from the $A$-coalgebra structure of $B(k, A, A)$~\cite{felix-halperin-thomas:1995} and from the morphism $\tilde{\mu}: \Delta^*(\Ad \otimes \Ad) \to \Ad$ of $A_\infty$-modules over $A$ of Proposition~\ref{prop:a-inf-mult-ad}. This $A_\infty$-algebra structure should be equivalent to that of $CH^*(A, A)$ under the equivalences of Chapter~\ref{ch:hh}.

%% file: sec-bv-cap-pairings.tex
\subsection{Gerstenhaber and BV Structures}\label{sec:bv-structs}

\subsubsection{Relating the Hochschild and \texorpdfstring{$\Ext$/$\Tor$}{Ext/Tor} cap products}

By analogy with the cup-pairing of McClure and Smith, we introduce the notion of a cap-pairing between simplicial and cosimplicial spaces (or spectra), modeled on the cap product of Hochschild cochains on chains.

\begin{defn}
Let $X^\bullet$ be a cosimplicial space, and let $Y_\bullet$ and $Z_\bullet$ be simplicial spaces. A \term{cap-pairing} $c: (X^\bullet, Y_\bullet) \to Z_\bullet$ is a family of maps
\[
  c_{p,q}: X^p \times Y_{p + q} \to Z_q
\]
satisfying the following relations:
\begin{align*}
	c_{p,q}(d^i f, x) &= c_{p - 1, q}(f, d_i x), && 0 \leq i \leq p, \\
	c_{p,q}(d^p f, x) &= d_0 c_{p - 1, q + 1}(f,x), \\
	c_{p - 1, q}(f, d_{p + i}x) &= d_{i + 1} c_{p - 1, q + 1}(f, x), && 0 \leq i < q, \\
	c_{p,q}(s^i f, x) &= c_{p + 1, q}(f, s_i x), && 0 \leq i \leq p, \\
	c_{p,q}(f, s_{p + i} x) &= s_i c_{p + 1, q - 1}(f,x), && 0 \leq i < q.
\end{align*}
Note all of these relations hold in the space $Z_q$.

A \term{morphism} of cap-pairings from $c: (X^\bullet, Y_\bullet) \to Z_\bullet$ to $c': ({X'}^\bullet, Y'_\bullet) \to Z'_\bullet$ is a triple consisting of a cosimplicial map $\mu_1: X \to X'$ and simplicial maps $\mu_2: Y \to Y'$ and $\mu_3: Z \to Z'$ such that
\[
  \mu_3 \circ c_{p,q} = c'_{p, q} \circ (\mu_1 \times \mu_2) \quad \text{for all $p, q$}. 
\]
Analogous constructions pertain to simplicial and cosimplicial spectra.
\end{defn}

Just as a cup-pairing $\phi: (X^\bullet, Y^\bullet) \to Z^\bullet$ induces a family of maps $\Tot X^\bullet \times \Tot Y^\bullet \to \Tot Z^\bullet$, a cap-pairing induces a map $\Tot X^\bullet \times |Y_\bullet| \to |Z_\bullet|$:

\begin{prop}
Let $c: (X^\bullet, Y_\bullet) \to Z_\bullet$ be a cap-pairing. Then for each $u$ with $0 < u < 1$, $c$ induces a map
\[
  \bar{c}_u: \Tot X^\bullet \times |Y_\bullet| \to |Z_\bullet|.
\]
A morphism $(\mu_1, \mu_2, \mu_3): c \to c'$ of cap-pairings induces a commuting diagram
\[
  \xymatrix{
	\Tot X^\bullet \times |Y_\bullet| \ar[d]_{\Tot \mu_1 \times |\mu_2|} \ar[r]^-{\bar{c}_u} & |Z_\bullet| \ar[d]^{|\mu_3|} \\
	\Tot X'^\bullet \times |Y'_\bullet| \ar[r]^-{\bar{c}'_u} & |Z'_\bullet|
}
\]
\begin{proof}
This follows from the same prismatic subdivision techniques used~\cite{mcclure-smith:2002} to produce the maps $\bar{\phi}_u: \Tot X^\bullet \times \Tot Y^\bullet \to \Tot Z^\bullet$ from a cup-pairing $\phi: (X^\bullet, Y^\bullet) \to Z^\bullet$. For $n \geq 0$, define
\[
  D^n = \left(\coprod_{p = 0}^n \Delta^p \times \Delta^{n - p} \right)/\sim,
\]
where $\sim$ denotes the identifications $(d^{p + 1}s, t) \sim (s, d^0 t)$ for $s \in \Delta^p$ and $t \in \Delta^{n - p - 1}$. For each $u$, let $\sigma^n(u):  D^n \to \Delta^n$ be defined on $(s, t) \in \Delta^p \times \Delta^{n - p}$ by
\[
  \sigma^n(u)(s,t) = (us_0, \dotsc, us_{p - 1}, us_p + (1 - u) t_0, (1 - u) t_1, \dotsc, (1 - u) t_{n - p}).
\]
Then for $0 < u < 1$, $\sigma^n(u)$ is a homeomorphism. We use $\sigma^n(u)$ to define the map $\bar{c}_u$. Take $f \in \Tot X^\bullet$ and $(s, y) \in |Y_\bullet|$, and recall that $f$ is a sequence $(f_0, f_1, \dotsc)$ of functions $f_n: \Delta^n \to X^n$ commuting with the cosimplicial structure maps of $\Delta^\bullet$ and $X^\bullet$. Suppose $s \in \Delta^{p + q}$ and $y \in Y_{p + q}$, and that $\sigma^n(u)^{-1}(s) = (s', s'') \in \Delta^p \times \Delta^q \subset D^{p + q}$. Then
\[
  \bar{c}_u(f, (s, y)) = (s'', c_{p, q}(f(s'), y)) \in \Delta^q \times Z_q.
\]
The properties in the definition of the cap-pairing ensure that this map is well-defined: the second face-coface relation shows that this map is well-defined if a different representative is taken for $\sigma^n(u)^{-1}(s)$, and the other relations show that the map is well-defined for different representatives of $(s, y) \in |Y_\bullet|$. 

The naturality of these constructions in the simplicial and cosimplicial objects then shows that a morphism of cap-pairings induces such a commuting diagram.
\end{proof}
\end{prop}

We now apply this cap-pairing framework to the simplicial and cosimplicial spectra above.

\begin{prop}
$S[G^c] \wedge_G EG_+$ is a right module for the ring spectrum $F_G(EG_+, S[G^c])$. Under the equivalences above, this module structure is equivalent to the $THH_S(S[G])$-module structure of $THH^S(S[G])$.
\begin{proof}
We first explain the module structure of $S[G^c] \wedge_G EG_+$ in terms of a cap-pairing between cosimplicial and simplicial spectra. Recall that $F_G(EG_+,  S[G^c]) = \Tot Z^\bullet$, where $Z^n = \Map_G(B_n(G, G, *), S[G^c])$, and that $S[G^c] \wedge_G EG_+ \simeq |W_\bullet|$, where $W_n = S[G^c] \wedge_G B_n(G, G, *)_+$. The cap-pairing is then a collection of compatible maps
\[
  c_{p,q}: Z^p \wedge W_{p + q} \to W_q
\]
given on elements $a \in Z^p$ and $s \wedge g[g_1 \mid \dotsc \mid g_{p + q}] \in W_{p + q}$ by
\[
  c_{p,q}(a, (s \wedge g[g_1 \mid \dotsc \mid g_{p + q}])) = s a(g[g_1 \mid \dotsc \mid g_p]) \wedge g g_1 \dotsm g_p [g_{p + 1} \mid \dotsc \mid g_{p + q}].
\]
As with the cup-pairing on $Z^\bullet$, this map comes from the simplicial diagonal on $B_*(G, G, *)$ composed with the Alexander-Whitney approximation, and then applying the map to the left factor of the diagonal. 

We show that this cap-pairing is compatible with the cup-pairing on $Z^\bullet$ giving rise to the ring structure on $F_G(EG_+, S[G^c])$, and in fact makes $S[G^c] \wedge_G EG_+$ a right module over $F_G(EG_+, S[G^c])$. Take $a \in Z^p, b \in Z^q$, and $c = s \wedge g[g_1 \mid \dotsc \mid g_{p + q + r}] \in W_{p + q + r}$. Then
\begin{align*}
  c_{q, r}(b, c_{p, q + r}(a, c))
  & = c_{q, r}(b, s a(g[g_1 \mid \dotsc \mid g_p]) \wedge g g_1 \dotsm g_p [g_{p + 1} \mid \dotsc \mid g_{p + q + r}]) \\
  &= s a(g[g_1 \mid \dotsc \mid g_p]) b(g g_1 \dotsm g_p [g_{p + 1} \mid \dotsc \mid g_{p + q}]) \\
  & \mbox{}\hspace{2in}\mbox{} \wedge g g_1 \dotsm g_{p + q} [g_{p + q + 1} \mid \dotsc \mid g_{p + q + r}]) \\
  &= s (a \cup b)(g[g_1 \mid \dotsc \mid g_{p + q}]) \wedge g g_1 \dotsm g_{p + q} [g_{p + q + 1} \mid \dotsc \mid g_{p + q + r}] \\
  &= c_{p + q, r}(a \cup b, c)
\end{align*}

Similarly, the right $THH_S(S[G])$-module structure of $THH^S(S[G])$ via the Hochschild cap product can be described in terms of these cap pairings. As above, we have that $\Tot Y^\bullet \simeq THH_S(S[G])$, where $Y^n = \Map(G^n, S[G])$, and that $|V_n| \simeq THH^S(S[G])$, where $V_n = S[G] \wedge_{G \times G^{\op}} B_n(G,G,G)_+$. Levelwise, $V_n = S[G] \wedge (G^n)_+$. Then there is a cap-pairing
\[
  h_{p, q}: Y^p \wedge V_{p + q} \to V_q
\]
given by
\[
  h_{p,q}(f, a \wedge [g_1 \mid \dotsc \mid g_{p + q}]) = a f([g_1 \mid \dotsb \mid g_p]) \wedge [g_{p + 1} \mid \dotsc \mid g_{p + q}].
\]
This cap-pairing thus comes from evaluating the $p$-cochain on the first $p$ factors of the $(p + q)$-chain. A simple calculation shows that this cap-pairing is compatible with the cup-pairing on $Y^\bullet$ and therefore induces the desired right $THH_S(S[G])$-module structure on $THH^S(S[G])$.

We now show that the isomorphisms of simplicial and cosimplicial spectra $\psi^\bullet: Z^\bullet \to Y^\bullet$ and $\chi_\bullet: W_\bullet \to V_\bullet$ are compatible with the cap-pairings $c$ and $h$. Hence, we check that $h_{p,q} \circ (\psi^p \wedge \chi_{p + q}) = \chi_q c_{p,q}$:
\begin{align*}
	h_{p,q}(\psi^p(f) \wedge \mbox{}&\chi_{p + q}(a \wedge [g_1 \mid \dotsb \mid g_{p + q}])) \\
	&= (g_1 \dotsm g_{p + q})^{-1} a \psi^p(f)([g_1 \mid \dotsb \mid g_p]) \wedge [g_{p + 1} \mid \dotsb \mid g_{p + q}] \\
	&= (g_1 \dotsm g_{p + q})^{-1} a f([g_1 \mid \dotsb \mid g_p])g_1 \dotsm g_p \wedge [g_{p + 1} \mid \dotsb \mid g_{p + q}] \\
	&= \chi_q((g_1 \dotsm g_p)^{-1} a f([g_1 \mid \dotsb \mid g_p]) g_1 \dotsm g_p \wedge [g_{p + 1} \mid \dotsb \mid g_{p + q}]) \\
	&= \chi_q(a f([g_1 \mid \dotsb \mid g_p]) \wedge g_1 \dotsm g_p[g_{p + 1} \mid \dotsb \mid g_{p + q}]) \\
	&= \chi_q(c_{p,q}(f \wedge (a \wedge [g_1 \mid \dotsb \mid g_{p + q}]))).
\end{align*}
Since this holds, the right action of $F_G(EG_+, S[G^c])$ on $S[G^c] \wedge_G EG_+$ is equivalent to that of $THH_S(S[G])$ on $THH^S(S[G])$.
\end{proof}
\end{prop}

Under the equivalences of Section~\ref{ssec:ring-spectrum-equivs}, these module structures should be equivalent to Klein's module structure of $S[LM]$ over the $A_\infty$-ring spectrum $LM^{-TM}$~\cite{klein:2006}.

Again applying $- \wedge Hk$ and passing to the derived category of chain complexes, we obtain that $\Ad(\Omega M) \otimes^L_{C_*\Omega M} k$ is a right $A_\infty$-module for $R\Hom_{C_*\Omega M}(k, \Ad(\Omega M))$, and that this module structure is equivalent to that of the Hochschild cochains acting on the Hochschild chains.%

We now relate these cap products to the evaluation map and to the isomorphism $D$. 

\begin{prop}
View $\eta: k \to \Ad(\Omega M)$ as a map of $C_*\Omega M$-modules, inducing a map $\eta_*: \Tor_*^{C_*\Omega M}(k, k) \to \Tor_*^{C_*\Omega M}(\Ad, k)$. Then for $z \in \Tor_*^{C_*\Omega M}(k, k)$, $f \in \Ext^*_{C_*\Omega M}(k, \Ad(\Omega M))$, 
\[
  \ev_z(f) = (-1)^{|f||z|} \eta_*(z) \cap f.
\]
\begin{proof}
By the form of the cap-pairing on the spectrum level, the cap product 
\[
  \Ext^*_{C_* \Omega M}(k, \Ad(\Omega M)) \otimes \Tor_*^{C_*\Omega M}(\Ad(\Omega M), k) \to \Tor_*^{C_*\Omega M}(\Ad(\Omega M), k)
\]
is given by the sequence of maps
\begin{multline*}
  \Ext^*_{C_*\Omega M}(k, \Ad) \otimes \Tor_*^{C_*\Omega M}(\Ad, k)
  \xrightarrow{\isom} \Ext^*_{C_*\Omega M}(k, \Ad) \otimes \Tor_*^{C_*\Omega M}(\Ad, \Delta^*(k \otimes k)) \\
  \xrightarrow{\ev} \Tor_*^{C_*\Omega M}(\Ad, \Delta^*(\Ad \otimes k))
  \xrightarrow{\isom} \Tor_*^{C_*\Omega M}(\Delta^*(\Ad \otimes \Ad), k)
  \xrightarrow{\tilde{\mu}_*} \Tor_*^{C_*\Omega M}(\Ad, k),
\end{multline*}
where $\tilde{\mu}$ is the morphism of $A_\infty$-modules over $C_*\Omega M$ given in Prop.~\ref{prop:a-inf-mult-ad}. By introducing an extra $k$ factor via $\Delta_k$ and then collapsing it via $\lambda$, $\ev: \Ext_{C_*\Omega M}^*(k, \Ad) \otimes \Tor_*^{C_*\Omega M}(k, k) \to \Tor_*^{C_*\Omega M}(\Ad, k)$ is similarly given by
\begin{multline*}
\Ext^*_{C_*\Omega M}(k, \Ad) \otimes \Tor_*^{C_*\Omega M}(k, k)
  \xrightarrow{\isom} \Ext^*_{C_*\Omega M}(k, \Ad) \otimes \Tor_*^{C_*\Omega M}(k, \Delta^*(k \otimes k)) \\
  \xrightarrow{\ev} \Tor_*^{C_*\Omega M}(k, \Delta^*(\Ad \otimes k))
  \xrightarrow{\isom} \Tor_*^{C_*\Omega M}(\Delta^*(k \otimes \Ad), k)
  \xrightarrow{\lambda_*} \Tor_*^{C_*\Omega M}(\Ad, k)
\end{multline*}
Then for a given $f \in \Ext^*_{C_*\Omega M}(k, \Ad(\Omega M))$, the sequence of squares
\[
  \xymatrix@C-1em{
	\Tor(k, k) \ar[r] \ar[d]_{\eta_*} & \Tor(k, \Delta^*(k \otimes k)) \ar[r]^{\ev(f)} \ar[d]_{\eta_*}  & \Tor(k, \Delta^*(\Ad \otimes k)) \ar[r] \ar[d]_{\eta_*} & \Tor(\Delta^*(k \otimes \Ad), k) \ar[d]_{\eta_*} \\
	\Tor(\Ad, k) \ar[r] & \Tor(\Ad, \Delta^*(k \otimes k)) \ar[r]^{\ev(f)} & \Tor(\Ad, \Delta^*(\Ad \otimes k)) \ar[r] & \Tor(\Delta^*(\Ad \otimes \Ad), k)
}
\]
commutes. Since $\mu$ is unital, $\mu (\eta \circ \Ad) = \lambda: \Delta^*(k \otimes \Ad) \to \Ad$ as maps of chain complexes. Hence, applying $\Tor^{C_*\Omega M}_*(-, k)$ to the composite
\[
  \Delta^*(k \otimes \Ad) \xrightarrow{\eta \otimes \Ad}
  \Delta^*(\Ad \otimes \Ad) \xrightarrow{\tilde{\mu}}
  \Ad
\]
gives $\Tor^{C_*\Omega M}_*(\lambda, k)$. Taking into account the swap between the $\Ext$ and $\Tor$ tensor factors for the cap product, this establishes the identity $\ev_z(f) = (-1)^{|f||z|} \eta_*(z) \cap f$. 
\end{proof}
\end{prop}

\begin{prop}
The isomorphism $D: HH^*(C_*\Omega M) \to HH_{* + d}(C_*\Omega M)$ is given by
\[
  D(f) = (-1)^{|f| d} z_H \cap f,
\]
where $z_H \in HH_d(C_*\Omega M)$ is the image of $[M] \in \Tor_d^{C_*\Omega M}(k, k)$ under the maps
\[
  \Tor_*^{C_*\Omega M}(k, k) \xrightarrow{\Tor(\eta, k)} \Tor_*^{C_*\Omega M}(\Ad(\Omega M), k) \xrightarrow{\Lambda_*^{-1}} HH_*(C_*\Omega M).
\]
\begin{proof}
By construction, $D(f) = \Lambda_*^{-1}(\ev_{[M]}(\Lambda^* f))$. By the above proposition,
\begin{align*}
  \Lambda_*^{-1}(\ev_{[M]}(\Lambda^* f))
  &= (-1)^{|f| d}\Lambda_*^{-1}(\eta_*[M] \cap \Lambda^* f) \\
  &= (-1)^{|f| d} \Lambda_*^{-1}(\eta_*[M]) \cap f
  = (-1)^{|f| d} z_H \cap f,
\end{align*}
so $D(f) = (-1)^{|f| d} z_H \cap f$.
\end{proof}
\end{prop}

\begin{prop}
$B(z_H) = 0$.
\begin{proof}
Observe that we have the following commutative diagram:
\[
  \xymatrix{
	H_*(M) \ar[d]^{c_*} \ar[r]^-{\isom} & HH_*(C_*\Omega M, k) \ar[r]^{\Lambda_*}_{\isom} & \Tor_*^{C_*\Omega M}(k, k) \ar[d]^{\Tor(\eta, k)} \\
	H_*(LM) \ar[r]^-{BFG}_-{\isom} & HH_*(C_*\Omega M, C_*\Omega M) \ar[r]^{\Lambda_*}_{\isom} & \Tor_*^{C_*\Omega M}(\Ad(\Omega M), k)
}
\]
where $c: M \to LM$ is the map sending $x \in M$ to the constant loop at $x$. Then $B(z_H) = BFG (\Delta(c_*[M]))$. The trivial action of $S^1$ on $M$ induces a degree-$1$ operator $\Delta$ on $H_*(M)$ that is identically $0$. Since $c$ is $S^1$-equivariant with respect to these actions, $\Delta \circ c_* = c_* \circ \Delta = 0$, so $B(z_H) = 0$. 
\end{proof}
\end{prop}

%% file: sec-bv-hh-bv-structure.tex
\subsubsection{The BV structures on \texorpdfstring{$HH^*(C_*\Omega M)$}{HH*(C*Omega M)} and String Topology}

Now that we have shown that $D$ arises as a cap product in Hochschild homology, we may employ an algebraic argument of Ginzburg~\cite{ginzburg:2006}, with sign corrections by Menichi~\cite{menichi:2009a}, to show that this gives $HH^*(C_*\Omega M)$ the structure of a BV algebra.

For any DGA $A$, the cup product on $HH^*(A)$ is graded-commutative, so the right cap-product action of $HH^*(A)$ on $HH_*(A)$ also defines a left action, with $a \cdot z = (-1)^{|a||z|} z \cap a$ for $z \in HH_*(A)$ and $a \in HH^*(A)$.  Hence, each $a \in HH^*(A)$ defines a degree-$|a|$ operator $i_a$ on $HH_*(A)$ by $i_a(z) = a \cdot z$. Then $D(a) = a \cdot z_H = i_a(z_H)$. 

Similarly, for each $a \in HH^*(A)$, there is a ``Lie derivative'' operator $L_a$ on $HH_*(A)$ of degree $|a| + 1$, and there is the Connes $B$ operator of degree $1$. It is well known that these operations make $(HH^*(A), HH_*(A))$ into a calculus, an algebraic model of the interaction of differential forms and polyvector fields on a manifold. Tamarkin and Tsygan~\cite{tamarkin-tsygan:2005} in fact extend this calculus structure to a notion of $\infty$-calculus on the Hochschild chains and cohains of $A$, which descends to the usual calculus structure on homology, and they provide explicit descriptions of the operations on the chain level. The Lie derivative in this calculus structure is the graded commutator
\[
  L_a = [B, i_a],
\]
which for $a, b \in HH^*(A)$ satisfies the relations
\[
  i_{[a,b]} = (-1)^{|a| + 1}[L_a, i_b]
  \quad \text{and} \quad
  L_{a \cup b} = L_a i_b + (-1)^{|a|} i_a L_b,
\]
where $[a,b]$ is the usual Gerstenhaber Lie bracket in $HH^*(A)$.

\begin{thm}\label{thm:bv-struct-hh}
$HH^*(C_*\Omega M)$ is a BV algebra under the Hochschild cup product and the operator $\kappa = -D^{-1}BD$. The Lie bracket induced by this BV algebra structure is the standard Gerstenhaber Lie bracket.
\begin{proof}
Recall that $D(a) = a \cdot z_H$, so $B(a \cdot z_H) = -\kappa(a) \cdot z_H$. Then
\begin{align*}
  D([a,b]) &= i_{[a,b]}(z_H) \\
  &= (-1)^{|a| + 1}(L_a i_b - (-1)^{(|a| - 1)|b|} i_b L_a)(z_H) \\
  &= (-1)^{|a| + 1}(B i_a i_b -\! (-1)^{|a|} i_a B i_b -\! (-1)^{(|a| - 1)|b|} i_b B i_a + \!(-1)^{(|a| - 1)|b| + |a|} i_b i_a B)(z_H) \\
  &= (-1)^{|a| + 1} B( (a \cup b) \cdot z_H ) + a \cdot B(b \cdot z_H) + (-1)^{|a||b| + |b| + |a|} b \cdot B(a \cdot z_H) \\
  &= (-1)^{|a|} \kappa(a \cup b) \cdot z_H - (a \cup \kappa(b)) \cdot z_H - (-1)^{|a||b| + |b| + |a|} (b \cup \kappa(a)) \cdot z_H \\
  &= ((-1)^{|a|} \kappa(a \cup b) - (-1)^{|a|} \kappa(a) \cup b - a \cup \kappa (b)) \cdot z_H
\end{align*}
so therefore
\[
  [a,b] = (-1)^{|a|} \kappa(a \cup b) - (-1)^{|a|} \kappa(a) \cup b - a \cup \kappa (b).
\]
Since $HH^*(C_*\Omega M)$ is a Gerstenhaber algebra under $\cup$ and $[\,,\,]$, this identity shows that it is a BV algebra under $\cup$ and $\kappa$.
\end{proof}
\end{thm}

\begin{thm}\label{thm:bv-isom-hh-string-top}
Under the isomorphism $BFG \circ D: HH^*(C_*\Omega M) \to H_{* + d}(LM)$, the BV algebra structure above coincides with the BV algebra structure of string topology.
\begin{proof}
We have seen that the isomorphism $HH^*(C_*\Omega M) \isom H_{* + d}(LM)$ coming from spectra coincides with the composite isomorphism $BFG \circ D$, and so the latter takes the Hoch\-schild cup product to the Chas-Sullivan loop product. Furthermore, 
\[
  BFG \circ D \circ \kappa = - BFG \circ B \circ D = - \Delta \circ BFG \circ D,
\]
so $BFG \circ D$ takes $\kappa$ to $-\Delta$, the negative of the BV operator on string topology.

Tamanoi gives an explicit homotopy-theoretic construction of the loop bracket and BV operator in string topology~\cite{tamanoi:2007b}. In his Section~5, he notes that the bracket associated to the usual $\Delta$ operator is actually the negative $-\{-,-\}$ of the loop bracket, as defined using Thom spectrum constructions. Consequently, $-\Delta$ should be the correct BV operator on $\mathbb{H}_*(LM)$, since the sign change carries through to give $\{-, -\}$ as the bracket induced from the BV algebra structure. Then the Hochschild Lie bracket $[-, -]$ does correspond to the loop bracket under this isomorphism.

We conclude that $BFG \circ D$ is an isomorphism of BV algebras from $(HH^*(C_*\Omega M), \cup, \kappa)$ to the string topology BV algebra $(H_{* + d}(LM), \circ, -\Delta)$.
\end{proof}
\end{thm}

We also compare this result to the previous BV algebra isomorphisms between string topology and Hochschild homology. We note that Vaintrob's argument in~\cite{vaintrob:2007} relies on Ginzburg's algebraic argument without Menichi's sign corrections. With those sign changes in place, the argument appears to carry through to produce $-D^{-1}BD$ as the appropriate BV operator on Hochschild cohomology, and thus to give $-\Delta$ as the BV operator on string topology.

As noted above, Felix and Thomas also construct a BV algebra isomorphism between $\mathbb{H}_*(LM)$ and $HH^*(C^*M)$ when $M$ is simply connected and when $k$ is a field of characteristic $0$~\cite{felix-thomas:2007}. They invoke results of Menichi on cyclic cohomology~\cite{menichi:2001} and of Tradler and Zeinalian~\cite{tradler:2008b,tradler-zeinalian-sullivan:2007} to state that their BV operator on $HH^*(C^*M)$ induces the Gerstenhaber Lie bracket. In light of the sign change above and the isomorphism of Gerstenhaber algebras $HH^*(C^*M) \isom HH^*(C_*\Omega M)$ of Felix, Menichi, and Thomas for $M$ simply connected, it would be of interest to trace through these isomorphisms to check the sign of the induced bracket in their context.

%% file: app-alg.tex
\section{Algebraic Structures}\label{app:algebra}

\input{sec-bg-rev-hopf-adjoints.tex}

\input{sec-bg-a-infty.tex}

%% file: sec-bg-rev-hopf-adjoints.tex
\subsection{Hopf Algebras and Adjoint Actions}\label{ssec:hopf-algs-adjoints}

\subsubsection{Conventions for Chain Complexes and Modules}

The symmetry map $\tau$ in $\Ch(k)$ introduces a Koszul sign:

\begin{defn}
For $A, B$ chain complexes, define the algebraic twist map $\tau_{A, B}: A \otimes B \to B \otimes A$ by $\tau_{A, B}(a \otimes b) = (-1)^{|a| |b|} b \otimes a$. If $A, B$ are clear from context, $\tau_{A, B}$ is written $\tau$.
\end{defn}

We use the following notation when using $\tau$ to permute several factors in a tensor product of complexes (or Cartesian product of spaces):

\begin{notation}\label{sec:twist-notation}
Suppose that $\sigma \in S_n$ is a permutation on $n$ letters $\{ 1, \dotsc, n \}$. Denote by $\tau_{n, \sigma}$, or $\tau_{\sigma}$ if $n$ is understood, the unique morphism $X_1 \otimes \dotsb \otimes X_n \to X_{\sigma^{-1}(1)} \otimes \dotsb \otimes X_{\sigma^{-1}(n)}$ composed of the $\tau_{X_i, X_j}$ and taking the $i$th factor in the source to the $\sigma(i)$th factor in the target. 
\end{notation}

Consequently, for $\rho, \sigma \in S_n$, $\tau_{\rho} \circ \tau_{\sigma} = \tau_{\rho \sigma}$.

We occasionally use \term{Sweedler notation}~\cite{sweedler:1969} when working with the coproduct of a differential graded coalgebra.

\begin{notation}
Suppose $C$ is a DGC, and take $c \in C$. We write the coproduct of $c$ as $\Delta(c) = \sum_c c^{(1)} \otimes c^{(2)}$, with the index $c$ indicating a sum over the relevant summands of $\Delta(c)$. By the coassociativity of $\Delta$, 
\[
  (\id \otimes \Delta)(\Delta(c)) = \sum_c c^{(1)} \otimes c^{(2,1)} \otimes c^{(2,2)}
  = \sum_c c^{(1,1)} \otimes c^{(1,2)} \otimes c^{(2)} = (\Delta \otimes \id)(\Delta(c)).
\]
We instead denote this twice-iterated coproduct unambiguously as $\Delta^2(c) = \sum_c c^{(1)} \otimes c^{(2)} \otimes c^{(3)}$. Higher iterates $\Delta^n(c)$ are denoted similarly, with components $c^{(1)}, \dotsc, c^{(n + 1)}$.
\end{notation}

\begin{exmp}
In Sweedler notation, the counital condition $\id = \rho(\id \otimes \epsilon)\Delta = \lambda(\epsilon \otimes \id)\Delta$ becomes
\[
  c = \sum_c c^{(1)}\epsilon(c^{(2)}) = \sum_c \epsilon(c^{(1)})c^{(2)}.
\]
\end{exmp}

We also introduce notation for ``pullback'' module functors.

\begin{notation}\label{notation:pullback-modules}
If $\phi: B \to A$ is a morphism of DGAs, then $\phi^*: \modulecategory{A} \to \modulecategory{B}$ is the functor taking an $A$-module to a $B$-module with action by $\phi$.
\end{notation}

Recall that $\phi^*$ is an exact functor. We will not distinguish between pullback functors for left and right modules, and instead will recognize them from context.

\subsubsection{Adjoint Actions}

Suppose that $A$ is a differential graded Hopf algebra with an algebra anti-automorphism $S: A \to A$, so that $S: A \to A^{\op}$ is an isomorphism of DGAs. Also assume that $S^2 = \id$.

\begin{defn}\label{defn:adjoint-maps}
For $\epsilon = 0, 1$, define $\ad_{\epsilon} = (\id \otimes S) \tau^\epsilon \Delta: A \to A \otimes A^{\op}$. Note that $\ad_0$ and $\ad_1$ are both DGA morphisms, since $\Delta$ and $\tau$ are. 
\end{defn}

Then if $M$ is an $A$-$A$-bimodule, it is canonically both a left and a right $A^e$-module, and so by pullback $\ad_0^* M$ and $\ad_1^* M$ are (possibly distinct) $A$-modules.

\subsubsection{Bar Resolution Isomorphisms}

Suppose now that $S$ is an antipode for $A$. Then the coinvariant module $k \otimes_A \ad^* A^e$ is isomorphic to $A$ as $A^e$-modules.

\begin{prop}\label{prop:dgh-env-alg-isoms}
If $A$ is a DGH with antipode $S$, then $\phi: \ad_0^*A^e \otimes_A k \to A$ given by $(a \otimes a') \otimes \lambda \mapsto \lambda aa'$ is an isomorphism of left $A^e$-modules, and $\phi': k \otimes_A \ad_1^*A^e \to A$ given by $\lambda \otimes (a \otimes a') = (-1)^{|a||a'|} \lambda a'a$ is an isomorphism of right $A^e$-modules. 
\end{prop}

In fact, these isomorphisms are induced from isomorphisms of $A^e$-resolutions for these modules, which we state below.

\begin{prop}\label{prop:dgh-complex-isos}
Let $A$ be a DGH with antipode $S$. There are simplicial isomorphisms 
\begin{align*}
	& \gamma^{L, 0}_\bullet: B_\bullet(A,A,A) \leftrightarrows B_\bullet(\ad_0^*A^e, A, k) : \phi^{L, 0}_n \\
	& \gamma^{R, 1}_\bullet: B_\bullet(A,A,A) \leftrightarrows B_\bullet(k, A, \ad_1^* A^e): \phi^{R, 1}_n
\end{align*}
which descend to isomorphisms on the corresponding realizations. When $S^2 = \id$, there are isomorphisms $(\gamma^{L, 1}_\bullet, \phi^{L, 1}_\bullet)$ and $(\gamma^{R, 0}_\bullet, \phi^{R, 0}_\bullet)$ in the opposite $\epsilon$-cases as well.
\begin{proof}
We first exhibit an isomorphism $\gamma^{L,0}_\bullet: B_\bullet(A, A, A) \to B_\bullet(\ad_0^*A^e, A, k)$ and its inverse:
\begin{align*}
  \gamma^{L, 0}_n(a[a_1\mid \dotsb\mid a_n]a') &= \pm (a \otimes (a_1\dotsm a_n)^{(2)} a')[a_1^{(1)} \mid \dotsb \mid a_n^{(1)}]  \\
  \phi^{L, 0}_n((b \otimes b')[b_1\mid \dotsb\mid b_n]) &= \pm b[b_1^{(1)} \mid \dotsb \mid b_n^{(1)}] S((b_1 \dotsm b_n)^{(2)}) b'.
\end{align*}
It is straightforward to verify that these are isomorphisms of simplicial $A^e$-modules and thus determine isomorphisms of the associated bar complexes.

Next, we show an isomorphism $\gamma^{R,1}_\bullet: B_\bullet(A, A, A) \to B_\bullet(k, A, \ad_1^* A^e)$ and its inverse:
\begin{align*}
  \gamma^{R, 1}_n(a'[a_1\mid \dotsb\mid a_n]a) &= \pm [a_1^{(2)} \mid \dotsb \mid a_n^{(2)}](a \otimes a'S((a_1\dotsm a_n)^{(1)}))  \\
  \phi^{R, 1}_n([b_1\mid \dotsb\mid b_n](b \otimes b')) &= \pm b'(b_1 \dotsm b_n)^{(1)}[b_1^{(2)} \mid \dotsb \mid b_n^{(2)}] b
\end{align*}

When $S^2 = \id$, we have the isomorphisms
\begin{align*}
  \gamma^{L, 1}_n(a[a_1\mid \dotsb\mid a_n]a') &= \pm (a \otimes (a_1\dotsm a_n)^{(1)} a')[a_1^{(2)} \mid \dotsb \mid a_n^{(2)}]  \\
  \phi^{L, 1}_n((b \otimes b')[b_1\mid \dotsb\mid b_n]) &= \pm b[b_1^{(2)} \mid \dotsb \mid b_n^{(2)}] S((b_1 \dotsm b_n)^{(1)}) b',
\end{align*}
which assemble to an isomorphism $\gamma^{L, 1}_\bullet: B_\bullet(A, A, A) \to B_\bullet(\ad_1^* A^e, A, k)$, and isomorphisms
\begin{align*}
  \gamma^{R, 0}_n(a'[a_1\mid \dotsb\mid a_n]a) &= \pm [a_1^{(1)} \mid \dotsb \mid a_n^{(1)}](a \otimes a'S((a_1\dotsm a_n)^{(2)}))  \\
  \phi^{R, 0}_n([b_1\mid \dotsb\mid b_n](b \otimes b')) &= \pm b'(b_1 \dotsm b_n)^{(2)}[b_1^{(1)} \mid \dotsb \mid b_n^{(1)}] b,
\end{align*}
which produce an isomorphism $\gamma^{R, 0}_\bullet: B_\bullet(A, A, A) \to B_\bullet(k, A, \ad_0^* A^e)$.
\end{proof}
\end{prop}

Relating these simplicial isomorphisms back to Prop.~\ref{prop:dgh-env-alg-isoms}, note that, for example, that $\ad_0^* A^e \otimes_A k$ is the cokernel of $d_0 - d_1: B_1(\ad_0^* A^e, A, k) \to B_0(\ad_0^* A^e, A, k)$, and that $A$ is the cokernel of $d_0 - d_1: B_1(A, A, A) \to B_0(A, A, A)$. Hence, the simplicial isomorphisms induces isomorphisms on these cokernels.

In fact, these simplicial isomorphisms hold for a Hopf object $H$ in an arbitrary symmetric monoidal category $(\mathcal{C}, \otimes, I)$, with a monoid anti-automorphism $S: H \to H$. For example, we apply this result to a topological group $G$ considered as a Hopf object in the category $\Top$ in Proposition~\ref{prop:top-bar-group-homeo}.

%% file: sec-bg-a-infty.tex
\subsection{\texorpdfstring{$A_\infty$}{A-infinity} Algebras and Modules}\label{sec:a-infty-algs}

\subsubsection{\texorpdfstring{$A_\infty$}{A-infinity} Algebras and Morphisms}

We recall briefly from Keller~\cite{keller:2001} the fundamental notions of such algebras and their modules, although we treat chain complexes homologically instead of cohomologically and therefore must reverse the signs of some degrees.

\begin{defn}\label{defn:a-infty-alg}
An $A_\infty$-algebra over $k$ is a graded $k$-module $A_*$ with a sequence of graded $k$-linear maps $m_n: A^{\otimes n} \to A$ of degree $n - 2$ for $n \geq 1$. These maps satisfy the quadratic relations
\[
  \sum_{\substack{r + s + t = n \\ s \geq 1}} (-1)^{r + st} m_{r + 1 + t}(\id^{\otimes r} \otimes m_s \otimes \id^{\otimes t}) = 0
\]
for $n \geq 1$, where $r, t \geq 0$ and $s \geq 1$.
\end{defn}

The first of these relations, $m_1 m_1 = 0$, shows that $m_1$ is a differential, making $A$ a chain complex. The second relation rearranges to
\[
  m_1 m_2 = m_2(m_1 \otimes \id + \id \otimes m_1),
\]
shows that $m_2: A \otimes A \to A$ is a chain map with respect to the differential $m_1$. The third identity rearranges to
\[
  m_2( \id \otimes m_2 - m_2 \otimes \id) = m_1 m_3 + m_3(m_1 \otimes \id^{\otimes 2} + \id \otimes m_1 \otimes \id + \id^{\otimes 2} \otimes m_1),
\]
which shows that $m_2$ is associative only up to chain homotopy, with $m_3$ the homotopy between the two different $m_2$ compositions. The higher relations then describe additional homotopy coherence data for the $m_n$ maps. Such data also describe a degree-$(-1)$ coderivation $b$ of the DGC $B(k, A, k)$ with $b^2 = 0$; for more details on both of these perspectives, see \cite[\S3]{keller:2001}. 

A differential graded algebra $A$ determines an $A_\infty$-algebra with $m_1 = d$, the differential of $A$, $m_2 = \mu$, and $m_n = 0$ for $n \geq 3$. Conversely, any $A_\infty$-algebra with $m_n = 0$ for $n \geq 3$ is a DGA. All of the $A_\infty$-algebras we consider will actually be DGAs.  Likewise, there is a notion of a morphism of $A_\infty$-algebras, but any morphism we consider between these DGAs will be an ordinary morphism of DGAs. In the coalgebra framework, a morphism of $A_\infty$-algebras $A \to A'$ is equivalent to a morphism of DGCs $B(k, A, k) \to B(k, A', k)$. 

\subsubsection{\texorpdfstring{$A_\infty$}{A-infinity} Modules and Morphisms}

We turn to the definition of modules over $A_\infty$ algebras and their morphisms.

\begin{defn}\label{defn:a-infty-module}
A (left) $A_\infty$-module over an $A_\infty$-algebra $A$ is a graded $k$-module $M$ with action maps $m^M_n:  A^{\otimes(n - 1)} \otimes M \to M$ of degree $n - 2$ for $n \geq 1$, satisfying the same relation as in Definition~\ref{defn:a-infty-alg}, with the $m_j$ replaced with $m^M_j$ where appropriate.
\end{defn}

This definition is equivalent to giving a degree-$(-1)$ differential $b_M$ with $b_M^2 = 0$ compatible with the left $B(k, A, k)$-comodule structure on $B(k, A, M)$. If $A$ is a DGA and $M$ is an ordinary $A$-module, then setting $m^M_1 = d_M$, $m^M_2 = a_M$, and $m^M_n = 0$ for $n \geq 3$ gives $M$ the structure of an $A_\infty$-module for $A$. All of the $A_\infty$-modules we consider will arise this way. 

We do need to consider morphisms of $A_\infty$-module which do not arise from morphisms of ordinary modules, however.

\begin{defn}\label{defn:a-infty-module-morphism}
Let $L, M$ be $A_\infty$-modules for an $A_\infty$-algebra $A$. A morphism $f: L \to M$ of $A_\infty$-modules over $A$ is a sequence of maps $f_n: A^{\otimes n - 1} \otimes L \to M$ of degree $n - 1$ satisfying the relations
\[
  \sum_{\substack{r + s + t = n \\ s \geq 1}} (-1)^{r + st} f_{r + t + 1} (\id^{\otimes r} \otimes m_s \otimes \id^{\otimes t}) 
  = \sum_{\substack{r + s = n \\ s \geq 1}} m_{r + 1} (\id^{\otimes r} \otimes f_s),
\]
for $n \geq 1$, where the $m_i$ represent the multiplication maps for the $A_\infty$-algebra $A$ or the action maps for $L$ and $M$.
\end{defn}

This definition is equivalent to specifying a morphism of DG comodules $B(L, A, k) \to B(M, A, k)$. While this perspective is convenient for more theoretical work, the explicit form of the maps and relations above is more suitable for checking that a proposed map is a morphism of modules.  

When $A$ is a DGA and $L$ and $M$ are $A$-modules, the $m_i$ vanish for $i \geq 3$, and we obtain the simplified relations $d_M f_1 = f_1 d_L$ and
\begin{multline*}
  d_M f_n + (-1)^n f_n d_{A^{\otimes n - 1} \otimes L} \\
  = - a_M (\id \otimes f_{n - 1}) + \sum_{r = 0}^{n - 3} (-1)^r f_{n - 1} (\id^{\otimes r} \otimes \mu \otimes \id^{\otimes n - r - 2}) + (-1)^{n - 2} f_{n - 1} (\id^{\otimes n - 2} \otimes a_L)
\end{multline*}
for $n \geq 2$. In this case, $f_1$ is a chain map $L \to M$ which commutes with the action of $A$ only up to a prescribed homotopy, $f_2$. Each subsequent $f_{n + 1}$ gives a homotopy between different ways of interleaving $f_n$ with the action of $n - 1$ copies of $A$. We will use these concepts in Section~\ref{sec:hh-adjoint-comp} when comparing different adjoint module structures over $C_*G$.

Gugenheim and Munkholm~\cite{gugenheim-munkholm:1974} note that $\Tor$ exhibits functoriality with respect to such morphisms, and Keller~\cite{keller:2001} notes that this functoriality also $\Ext_A^*(-, -)$ also exhibits such extended functoriality. We state the form of the results we need below:

\begin{prop}\label{prop:a-infty-extended-functoriality}
Let $A$ be a DGA, and let $L, M, N$ be $A$-modules. Suppose that $A$, $L$, and $M$ are all cofibrant as chain complexes of $k$-modules. Then a morphism of $A_\infty$-modules $f: L \to M$ induces maps
\[
  \Tor^A_*(N, f): \Tor^A_*(N, L) \to \Tor^A_*(N, M)
  \quad \text{and} \quad
  \Ext_A^*(N, f): \Ext_A^*(N, L) \to \Ext_A^*(N, M).
\]
If the chain map $f_1: L \to M$ is a quasi-isomorphism, these induced maps on $\Tor$ and $\Ext$ are also isomorphisms.
\begin{proof}
Since $f$ is an $A_\infty$-module morphism, it induces a morphism of $B(k, A, k)$-comodules $B(k, A, L) \to B(k, A, M)$. Since $B(A, A, L)$ can be described as the cotensor product
\[
  B(A, A, k) \square^{B(k, A, k)} B(k, A, L),
\]
this morphism of comodules induces a chain map
\[
  B(A, A, L) \to B(A, A, M).
\]
Since $A$, $L$, and $M$ are cofibrant over $k$, these bar constructions provide cofibrant replacements for $L$ and $M$ as $A$-modules. Applying the functors $N \otimes_A -$ and $\Hom_A(QN, -)$ and passing to homology then induces the desired maps on $\Ext$ and $\Tor$.

If $f_1$ is a quasi-isomorphism, then $B(A, A, f)$ is a weak equivalence between cofibrant $A$-modules and is therefore a homotopy equivalence. It therefore induces isomorphisms in $\Ext$ and $\Tor$.
\end{proof}
\end{prop}